\documentclass[letterpaper]{article}
\usepackage{uai2020}
\usepackage{times} 
\usepackage[margin=0.95in]{geometry}

\usepackage{times}

\title{Kernel Conditional Moment Test via Maximum Moment Restriction}

\author{ {\bf Krikamol Muandet} \\
MPI for Intelligent Systems \\
T\"ubingen, Germany \\
\And
{\bf Wittawat Jitkrittum\thanks{Now with Google Research.}}  \\
MPI for Intelligent Systems \\
T\"ubingen, Germany \\
\And
{\bf Jonas M. K\"ubler}   \\
MPI for Intelligent Systems \\
T\"ubingen, Germany
}

\usepackage{microtype}
\usepackage{graphicx}
\usepackage{booktabs} % for professional tables
\usepackage[numbers,sort&compress]{natbib}
\usepackage{cite}

\newcommand{\ie}{\emph{i.e.}} 
\newcommand{\eg}{\emph{e.g.}}

\usepackage{amsmath,amsthm,amssymb}
\usepackage{color}
\usepackage{bm}
\usepackage{bbm}
\usepackage{algorithm}
\usepackage{algorithmic}
\usepackage[inline]{enumitem}
\usepackage{mathtools}
\usepackage[toc,page,header]{appendix}
\usepackage{pgfplots}
\usepackage{mathrsfs}  
\usepackage{multicol}
\usepackage{multirow}
\usepackage{booktabs}
\pgfplotsset{compat=newest}
\usepackage{thmtools}
\usepackage{thm-restate}
\usepackage{subcaption}
\usepackage{booktabs}
\usepackage{verbatim}
\usepackage[inline]{enumitem}

\usepackage{minitoc}

\allowdisplaybreaks

\newtheorem{theorem}{Theorem}[section]
\newtheorem{definition}{Definition}[section]
\newtheorem{lemma}{Lemma}[section]
\newtheorem{corollary}{Corollary}[section]
\newtheorem{proposition}{Proposition}[section]

\usepackage{hyperref}
\hypersetup{
    colorlinks = true,
    citecolor = blue!60!black,
    linkcolor = blue!60!black
}

\newcommand{\inx}{\ensuremath{\mathcal{X}}}
 
\newcommand{\inz}{\ensuremath{\mathcal{Z}}}

\newcommand{\hbspf}{\ensuremath{\mathcal{F}}}

\newcommand{\rr}{\mathbb{R}} 		         % the real numbers
\newcommand{\ep}{\mathbb{E}}                     % the expectation

\newcommand{\mut}{\bm{\mu}}
\newcommand{\muth}{\widehat{\mut}}
\newcommand{\bxi}{\bm{\xi}}
\newcommand{\bpsi}{\bm{\psi}}

\newcommand{\dd}{\, \mathrm{d}}

\newcommand{\CMR}{\ensuremath{\mathscr{M}}}
\newcommand{\MM}{\mathbb{M}}
\newcommand{\MH}{\widehat{\MM}}

\newcommand{\x}{\ensuremath{\mathbf{x}}}

\begin{document}
 
\maketitle

\begin{abstract}
    We propose a new family of specification tests called kernel conditional moment (KCM) tests. 
    Our tests are built on a novel representation of conditional moment restrictions in a reproducing kernel Hilbert space (RKHS) called conditional moment embedding (CMME).
    After transforming the conditional moment restrictions into a continuum of unconditional counterparts, 
    the test statistic is defined as the maximum moment restriction (MMR) within the unit ball of the RKHS. 
    We show that the MMR not only fully characterizes the original conditional moment restrictions, leading to consistency in both hypothesis testing and parameter estimation, but also has an analytic expression that is easy to compute as well as closed-form asymptotic distributions. 
    Our empirical studies show that the KCM test has a promising finite-sample performance compared to existing tests.
\end{abstract}

\section{INTRODUCTION}

Many problems in causal inference, economics, and finance are often formulated as a conditional moment restriction (CMR): for correctly specified models, the conditional mean of certain functions of data is almost surely equal to zero \citep{Newey93:CMR,AI03:CMR}.
Rational expectation models---widely used in many fields of macroeconomics---specify how economic agents exploit available information to form their expectations in terms of conditional moments \citep{Muth61:Rational}.
Recent advances in causal machine learning also rely on the CMR including a generalized random forest (GRF)
\citep{Athey19:GRF}, orthogonal random forest (ORF) \citep{Oprescu19:ORF}, double machine learning (DML) \citep{Chernozhukov18:DML}, and nonparametric instrumental variable regression \citep{Bennett19:DeepGMM,Lewis18:AGMM} among others; see also \citet{Hartford17:DIV,Singh19:KIV,Muandet19:DualIV} and references therein.
%develops the  for estimating heterogeneous treatment effects based on conditional moments with unknown nuisance function \citep{AI03:CMR}.
%When there exist high-dimensional nuisance parameters, the Neyman orthogonality of the moment conditions has also been used to develop the two-stage algorithm known as .

%%%%
\begin{figure}[t!]
    \centering
    \resizebox{\columnwidth}{!}{
        \begin{tikzpicture}[scale=0.7]
      \begin{scope}
      % Input space
      \begin{axis}[axis y line=middle, axis line style = ultra thick, ymin=-1, ymax=1, xmin=-2, xmax=2, yticklabels={,,}, xticklabels={,,}, title={\Huge$\ep[\bpsi(Z;\theta)|X]$}]

        % Distribution on X
        \addplot[dashed,color=black,fill=yellow!50,opacity=0.5,smooth, line width=4pt, domain=-2:2] plot (\x,{-0.2+exp(-\x*\x*0.15/0.4)});
        \node at (-0.35,0.5) {\huge$P_X$};

        % Random functions
        \addplot[color=red,smooth,domain=-2:2,line width=3pt] plot (\x,{-0.65+1/(1+exp(3*\x))}) node[below,pos=1,xshift=-1em] {\huge $\theta_2$};
        \addplot[color=blue,smooth,domain=-2:2,line width=3pt] plot (\x,{-0.4+1/(1+exp(-\x))}) node[above,pos=1,xshift=-1em] {\huge $\theta_1$};;
        \addplot[color=green,smooth,domain=-2:2,line width=3pt] plot (\x,{0}) node[below,pos=1,xshift=-1em] {\huge $\theta_0$};;
      \end{axis}
      \end{scope}

      \begin{scope}[xshift=11cm]      
	\begin{axis}[axis y line=middle, axis x line=middle, axis line style = ultra thick, 
			ymin=-2, ymax=2, xmin=-2, xmax=2, yticklabels={,,}, xticklabels={,,}, title={\Huge RKHS $\mathcal{F}$}]	
      	% set up locations
      	\coordinate (P1) at (2,0.65);
      	\coordinate (P2) at (4.2,0.55);
      	\coordinate (Q1) at (2,0);   
      	\coordinate (Q2) at (5,0);
      	\coordinate (R1) at (2,-0.7);   
      	\coordinate (R2) at (6,-0.75);

	\node[label={45:{\huge $\mu_{\theta_0}$}},circle,fill=green,inner sep=5pt] at (axis cs:0,0) {};
	\node[label={45:{\huge $\mu_{\theta_1}$}},circle,fill=blue,inner sep=5pt] at (axis cs:-1.3,0.7) {};
	\node[label={45:{\huge $\mu_{\theta_2}$}},circle,fill=red,inner sep=5pt] at (axis cs:1.1,-1.5) {};
      
	\end{axis}
	\end{scope}
	\draw[dashed, line width=2pt, ->] (7.8,3) -- node[above, text width=4cm, align=center]
    {\bfseries CMME} (10.5,3);
    \end{tikzpicture}}
    \caption{\textbf{Conditional moment embedding (CMME)}: The conditional moments $\ep[\bm{\psi}(Z;\theta)|X]$ for different parameters $\theta$ are \emph{uniquely} ($P_X$-almost surely) embedded into the RKHS.
    The RKHS norm of $\mut_\theta$ measures to what extent these restrictions are violated and hence is used as a test statistic for conditional moment tests.}
    \label{fig:cmme}
\end{figure}
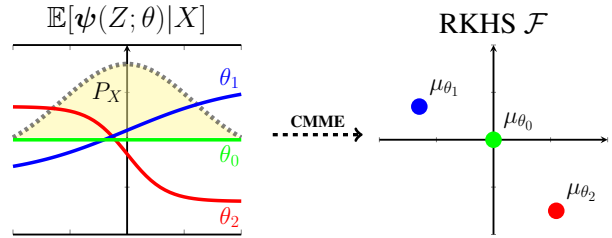
%%%%

%The specification testing aims to detect \emph{model misspecifications}: discrepancies between the models and the phenomena under study. 
Checking the validity of these moment restrictions is the first and foremost step to ensure that a model is correctly specified which constitutes a fundamental assumption for its estimation and inference.
A model misspecification often creates biases to parameter estimates, inconsistency of standard errors, and invalid asymptotic distributions that hinder our subsequent inference based on the model.
%Furthermore, hidden confounders that hinder most of the causal analyses may be detected using specification tests.
An overidentifying restriction test in the generalized method of moments (GMM) framework is one of the standard approaches to test a \emph{finite} number of \emph{unconditional} moment conditions \citep{Hansen82:GMM,Hall05:GMM}.
The $J$-test is an example of such tests \citep{Sargan58:OITest,Hansen82:GMM}, and numerous tests have been developed in econometrics to deal with various sources of misspecification; see, \eg, \citet{Bierens17:Econometric} for a review.
This paper focuses on an important class of CMR-based specification tests known as the conditional moment (CM) tests \citep{Newey85:Specification,Tauchen85:CMT} which have a long history in econometrics \citep{Hausman78:Specification,White81:Misspecification,Bierens17:Econometric}.

Testing \emph{conditional} moment restrictions becomes more challenging as an \emph{infinite} number of equivalent unconditional moment restrictions (UMR) must be examined simultaneously (cf. Section \ref{sec:mmr}).
At first, \citet{Newey85:Specification} and \citet{Tauchen85:CMT} proposed to perform the overidentifying restriction test on a finite subset of the UMR. 
Unfortunately, the CM tests that rely only on a finite number of moment conditions cannot be consistent against all alternatives.
Additional assumptions such as the global identification of selected moment conditions and sample-size dependent moment conditions are required to guarantee consistency \citep{Jong96:Dependence,Donald03:Likelihood}.
To overcome this limitation, \citet{Bierens82:ConsistentCM} introduced the first consistent CM tests---known as integrated conditional moment (ICM) tests---by checking \emph{all} moment conditions simultaneously \citep{Bierens97:Asymptotic}.
However, the ICM test depends on parametric weighting functions and nuisance parameters that limit its practical use.
An alternative class of consistent CM tests, known as smooth tests, employ nonparametric kernel estimation \citep{Zheng96:NonparamTest, LiWang98:NonparamTest} which also forms a basis for the generalized empirical likelihood approach \citep{Delgado06:Consistent,Tripathi03:CMRTest}.
However, they have non-trivial power only against local alternatives that approach the null at a slower rate than $1/\sqrt{n}$, and are susceptible to the curse of dimensionality (cf. Section \ref{sec:related-work} for the discussion).

Inspired by a surge of kernel-based tests \citep{gretton2012kernel,Chwialkowski16:KGFT,Liu16:KSD}, we propose to embed the CMR in a reproducing kernel Hilbert space (RKHS). 
By transforming CMR into a continuum of UMR in RKHS, the test statistic is defined as the maximum moment restriction (MMR) within the unit ball of the RKHS (cf. Section \ref{sec:mmr}).
We then show that the MMR corresponds to the RKHS norm of a Hilbert space embedding of conditional moments. 
\emph{Not only can the MMR capture all information about the original CMR, but it also has a closed-form expression that enables the practical ease of implementation} (cf. Theorems \ref{thm:identifiability} and \ref{thm:close-form}).
The MMR allows us to develop a class of consistent CM tests that we call kernel conditional moment (KCM) tests (cf. Section \ref{sec:kcm}).
Furthermore, it considerably simplifies the parameter estimation problems based on the CMR.
Our framework has relationships to existing methods in econometrics and machine learning (cf. Section \ref{sec:related-work}).
To the best of our knowledge, the Hilbert space embedding of conditional moment restrictions has not appeared elsewhere in the literature.\footnote{\citet{Carrasco00:Continuum} and their follow-up work are the most relevant works from the econometric literature. 
We discuss this connection in Section \ref{sec:related-work}.}

%The rest of this paper is organized as follows: Section \ref{sec:background} introduces the conditional moment restrictions and RKHS. 
%Next, Section \ref{sec:mmr} provides a novel representation of the conditional moment restrictions in RKHS, followed by the kernel conditional moment tests in Section \ref{sec:kcm}. We discuss related work in Section \ref{sec:related-work} and report empirical results in Section \ref{sec:experiments}. 
%Finally, we conclude the paper and discuss future research direction in Section \ref{sec:discussion}.

All proofs can be found in Appendix \ref{sec:proofs}. 
The code of our experiments is available at \url{https://github.com/krikamol/kcm-test}.

%%%%%
\section{BACKGROUND}
\label{sec:background}

We introduce the CMR in Section \ref{sec:cmt} and then review the concepts of kernels and RKHS in Section \ref{sec:rkhs}. 
Finally, we discuss the main assumptions in Section \ref{sec:assumption}. 

%%%
\subsection{CONDITIONAL MOMENT RESTRICTIONS}
\label{sec:cmt}

Let $Z$ be a random variable taking values in $\mathcal{Z}\subseteq\rr^p$ with distribution $P_Z$, $X$ a subvector of $Z$ taking values in $\inx\subseteq\rr^d$ with distribution $P_X$, and $\Theta\subset\rr^r$ a parameter space.
Following \citet{Newey93:CMR}, we consider models where the only available information about the unknown parameter $\theta_0\in\Theta$ is a set of conditional moment restrictions
\begin{equation}\label{eq:cond-moments}
    \mathscr{M}(X;\theta_0) := \ep[\bpsi(Z;\theta_0)|X] = \mathbf{0}, \quad P_X\text{-a.s.} ,
\end{equation}
\noindent where $\bpsi: \inz\times\Theta\to\rr^q$ is a vector of \emph{generalized residual functions} whose functional forms are known up to the parameter $\theta\in\Theta$.
The expectation is always taken over all random variables that are not conditioned on.
Note that there can be two different models that are \emph{observationally equivalent} on the basis of \eqref{eq:cond-moments} alone although an ideal parameter $\theta_0$ is unique.

Several statistical problems can be formulated as \eqref{eq:cond-moments}. 
In nonparametric regression models, $Z=(X,Y)$ where $Y\in\rr$ is a dependent variable and $\bpsi(Z;\theta) = Y-f(X;\theta)$.
For conditional quantile models, $Z=(X,Y)$ and $\bpsi(Z;\theta) = \mathbbm{1}\{Y < f(X;\theta)\}-\tau$ for the target quantile $\tau\in [0,1]$.
In heterogeneous effect estimation, $Z=(X,T,Y)$ where $T$ is a vector of treatments and $\bpsi(Z;\theta(X)) = (Y-\langle\theta(X),T\rangle)T$. 
For instrumental variable regression, $Z=(X,W,Y)$ where $W$ is an instrumental variable and $\bpsi(Z;\theta) = (Y-f_\theta(X))$ and $\ep[\bpsi(Z;\theta)|W] =  0$ almost surely.
When $Z$ admits the density $p(z;\theta)$, we can define the moment conditions in terms of the \emph{score function} as $\bpsi(Z;\theta) = \nabla_{\theta}\log p(Z;\theta)$ and use it for local maximum likelihood estimation.

%Although $\mathcal{M}$ and $\mathcal{M}_A$ are different, they are \emph{observationally equivalent} on the basis of \eqref{eq:cond-moments} alone.
%On the other hand, if the true model, say $M_B$, does not have such property, \ie, there exists no $\theta$ such that \eqref{eq:cond-moments} holds, then $\mathcal{M}$ and $\mathcal{M}_B$ have different implications for \eqref{eq:cond-moments}, leading to different asymptotic behaviors.

\vspace{-10pt}
\paragraph{Conditional moment tests.}
Given an independent sample $(x_i,z_i)_{i=1}^n$ drawn from a distribution that satisfies the conditional moments \eqref{eq:cond-moments} and an estimate $\hat{\theta}$ of $\theta_0$, our goal is to perform specification testing: \emph{Given a function $\bpsi$ and a parameter estimate $\hat{\theta}$, we test the null hypothesis}
\begin{equation}\label{eq:cm-test}
    H_0: \;\ep[\bpsi(Z;\hat{\theta})\,|\,X] = \mathbf{0}, \;\; P_X\text{-a.s.}
\end{equation}
For instance, in the test of functional form of the nonlinear regression model \citep{Hausman78:Specification}, the null hypothesis can be expressed as $H_0: \ep[Y - f(X;\hat{\theta})|X] = 0$ where $\hat{\theta} = \arg\min_{\theta\in\Theta}\,\ep[(Y - f(X;\theta))^2]$. 
In this case, $Z=(Y,X)$ and $\bpsi(Z;\theta) = Y-f(X;\theta)$.
This test allows us to detect misspecifications of the functional form of $f$.

In this work, we assume that $\hat{\theta}$ is obtained independently of the data that is used to test \eqref{eq:cm-test}.
In many cases, however, $\hat{\theta}$ is estimated using this data and hence the test performance is also subject to the estimation error.
A generalization of our framework to those cases will require more involved analyses, and we leave it to future work.

%If $\mathcal{M}$ be the underlying model which has the property that \eqref{eq:cond-moments} holds for some unique $\theta_0\in\Theta$.
%Note that there can be two alternative scenarios for misspecification. 
%First, the true model $\mathcal{M}_A$ is different from $\mathcal{M}$, but shares the property that \eqref{eq:cond-moments} holds for some unique $\theta'_0\in\Theta$.
%Although $\mathcal{M}$ and $\mathcal{M}_A$ are different, they are \emph{observationally equivalent} on the basis of \eqref{eq:cond-moments} alone.
%On the other hand, if the true model, say $M_B$, does not have such property, \ie, there exists no $\theta$ such that \eqref{eq:cond-moments} holds, then $\mathcal{M}$ and $\mathcal{M}_B$ have different implications for \eqref{eq:cond-moments}, leading to different asymptotic behaviors.

%%%
\subsection{REPRODUCING KERNELS}
\label{sec:rkhs}

Let $\inx$ be a non-empty set and $\hbspf$ a Hilbert space consisting of functions on $\inx$ with $\langle \cdot,\cdot\rangle_{\hbspf}$ and $\|\cdot\|_{\hbspf}$ being its inner product and norm, respectively.
The Hilbert space $\hbspf$ is called a reproducing kernel Hilbert space (RKHS) if there exists a symmetric function $k:\inx\times\inx\to\rr$ called the reproducing kernel of $\hbspf$ such that 
\begin{enumerate*}[label=(\roman*)]
   \item $k(x,\cdot)\in\hbspf$ for all $x\in\inx$ and 
   \item $f(x) = \langle f,k(x,\cdot)\rangle_{\hbspf}$ for all $f\in\hbspf$ and $x\in\inx$.
\end{enumerate*}
The latter is called the \emph{reproducing property} of $\hbspf$.
Every positive definite kernel $k$ uniquely determines the RKHS for which $k$ is a reproducing kernel \citep{aronszajn50reproducing}.

Let $\{(\lambda_j,e_j)\}$ be pairs of positive eigenvalues and orthonormal eigenfunctions of $k$, \ie, $\int e_i(x)e_j(x) \dd x = 1$ if $i=j$ and zero otherwise.
By Mercer's theorem \citep[Thm 4.49]{SteChr2008}, the kernel $k$ has the spectral decomposition
\begin{equation}\label{eq:mercer}
    k(x,x') = \sum_j\lambda_j e_j(x)e_j(x'), \quad x,x'\in\inx,
\end{equation}
where the convergence is absolute and uniform. 
As a result, for any $f\in\hbspf$, we have $f(x) = \sum_j f_je_j(x)$ with $\sum_{j}f_j^2/\lambda_j < \infty$ where $f_j = \langle f,e_j\rangle_{\hbspf}$, $\langle f,g\rangle_{\hbspf} = \sum_jf_jg_j/\lambda_j$, and $\|f\|_{\hbspf}^2 = \langle f,f \rangle_{\hbspf} = \sum_j f_j^2/\lambda_j$.

Next, we introduce the notion of integrally strictly positive definite (ISPD) kernels and Bochner's characterization.
\begin{definition}\label{def:integral}
    A kernel $k(x,x')$ is integrally strictly positive definite (ISPD) if for any function $f$ that satisfies $0<\|f\|_2^2<\infty$, $$\int_\inx f(x)k(x,x')f(x')\dd x\dd x' > 0 .$$
\end{definition}

ISPD kernels are an important notion in kernel methods and are closely related to characteristic and universal kernels, see, \eg, \citet{Simon-Gabriel18:KDE}.

The next result characterizes shift-invariant kernels $k(x,x') = \varphi(x-x')$ for some positive definite $\varphi$.

\begin{theorem}[Bochner]\label{thm:bochner}
    A continuous function $\varphi:\rr^d \to \mathbb{C}$ is positive definite if and only if it is the Fourier transform of a finite nonnegative Borel measure $\Lambda$ on $\rr^d$:
    $$\varphi(t) = \frac{1}{(2\pi)^{d/2}} \int_{\rr^d}e^{-i t^\top\omega}\dd\Lambda(\omega)$$ for $t\in\rr^d$.
\end{theorem}

Examples of popular kernels are the Gaussian RBF kernel $k(x,x') = \exp(-\|x-x'\|_2^2/2\sigma^2), \sigma > 0$, Laplacian kernel $k(x,x') = \exp(-\|x-x'\|_1/\sigma), \sigma > 0$, and inverse multiquadric (IMQ) kernel $k(x,x') = (c^2 + \|x-x'\|_2^2)^{-\gamma}$, $c,\gamma > 0$.
See, \eg, \citet[Ch. 4]{SteChr2008} for more examples.

%%%%
\subsection{MAIN ASSUMPTIONS}
\label{sec:assumption}

Our subsequent analyses rely on these key assumptions.

%\vspace{-0.5em}
\begin{enumerate}[label=\textnormal{\textbf{(A\arabic*)}},noitemsep]
    
\item \label{asmp:a1} The random vector $(X,Z)$ forms a strictly stationary process with the probability measure $P_{\mathit{XZ}}$.

\item \label{asmp:a2}
 \emph{Regularity conditions}: (i) the function $\bpsi:\inz\times\Theta\to\rr^q$ where $q < \infty$  is continuous on $\Theta$ for each $z\in\inz$; (ii) $\ep[\bpsi(Z;\theta)|x]$ exists and is finite for every $\theta\in\Theta$ and $x\in\inx$ for which $P_X(x) > 0$; (iii) $\ep[\bpsi(Z;\theta)|x]$ is continuous on $\Theta$ for all $x\in\inx$ for which $P_X(x)>0$.

\item \label{asmp:a3} \emph{Global identification}: there exists a unique $\theta_0\in\Theta$ for which $\ep[\bpsi(Z;\theta_0)|X]=\mathbf{0}$ a.s., and $P(\ep[\bpsi(Z;\theta)|X] = \mathbf{0}) < 1$ for all $\theta\in\Theta,\theta\neq \theta_0$.
 
\item \label{asmp:a4} The kernel $k$ is ISPD, continuous, and bounded, \ie, $\sup_{x\in\inx}\sqrt{k(x,x)} < \infty$.
 
%\item \label{asmp:a5} The parameter space $\Theta$ is compact.

%\item \label{asmp:a6} $\ep[\sup_{\theta\in\Theta}\|\bpsi(Z;\theta)\|] < \infty$.
\end{enumerate}

Assumption \ref{asmp:a1} ensures that all expectations of functions of $(X,Z)$ are independent of time.
The regularity conditions \ref{asmp:a2} are standard assumptions \citep[Ch. 3]{Hall05:GMM} which ensure that $\bpsi$ is well-defined, and hold in most models considered in the literature \citep{Hall05:GMM}.
By contrast, \ref{asmp:a3} may not hold, especially in non-linear models. 
A \emph{local} identifiability can be assumed instead by imposing additional constraints on $\Theta$. 
Testing whether the constraints are sufficient can then be done, for example, by examining the Jacobian at some parameter values \citep[pp. 54]{Hall05:GMM}.
Lastly, \ref{asmp:a4} implies that the RKHS $\hbspf$ consists of bounded continuous functions \citep[Sec. 4.3]{SteChr2008} and is expressive enough (cf. Theorem \ref{thm:identifiability}).

%%%%%
\section{MAXIMUM MOMENT RESTRICTION}
\label{sec:mmr}

This section presents the RKHS representation of the CMR.
Let $\mathscr{F}$ be a set of measurable functions on $\inx$.
Then, $\ep_{\mathit{XZ}}[\bpsi(Z;\theta)f(X)] = \ep_X[\ep_{\mathit{Z}}[\bpsi(Z;\theta)f(X)|X]] = \ep_X[\mathscr{M}(X;\theta)f(X)]$ for any $f\in\mathscr{F}$ by the law of iterated expectation.
That is, the CMR in \eqref{eq:cond-moments} implies an infinite set of unconditional moment restrictions
\begin{equation}\label{eq:uncond-moments}
    \ep[\bpsi(Z;\theta_0)f(X)] = \mathbf{0}, \quad \forall f\in\mathscr{F}.
\end{equation}
Equivalently, any $\theta_0\in\Theta$ that satisfies \eqref{eq:uncond-moments} must also satisfy what we call a \emph{maximum moment restriction} (MMR)
\begin{equation}\label{eq:sup-moments}
    \sup_{f\in\mathscr{F}}\; \|\ep[\bpsi(Z;\theta_0)f(X)]\|_2^2 = 0.
\end{equation}
It is known that the implied moment restrictions \eqref{eq:uncond-moments} and \eqref{eq:sup-moments} can be insufficient to globally identify the parameters of interest.
We call $\mathscr{F}$ for which \eqref{eq:sup-moments} implies \eqref{eq:cond-moments} a \emph{sufficient class of instruments}.
In the context of this work, $\mathscr{F}$ must consist of infinitely many instruments for the CM test to be consistent against all alternatives.
However, the $\sup$ operator also makes it hard to optimize \eqref{eq:sup-moments}.
We resolve these issues by choosing $\mathscr{F}$ to be a unit ball in a RKHS, which we show to be a sufficient class of instruments. 
As a result, \eqref{eq:sup-moments} can be solved analytically, the parameters of interest can be consistently estimated, and the resulting CM test is consistent against all fixed alternatives.

Recently, \citet{Lewis18:AGMM} and \citet{Bennett19:DeepGMM} also propose to estimate $\theta_0$ based on \eqref{eq:sup-moments} and $\mathscr{F}$ that is parameterized by deep neural networks. 
While they consider an estimation problem, we focus on hypothesis testing problems. 
Nevertheless, our formulation of CMR can also be used to estimate $\theta_0$ (cf. Section \ref{sec:kmmr} and Appendix \ref{sec:estimation}). 
Note that the algorithms proposed in \citet{Lewis18:AGMM} and \citet{Bennett19:DeepGMM} require solving a minimax game, whereas our approach for estimation is simply a minimization problem.

%%%%
\subsection{CONDITIONAL MOMENT EMBEDDING}
\label{sec:mmr-rkhs}

To express \eqref{eq:sup-moments} using the RKHS, we first develop a representation of the CMR in a vector-valued RKHS of functions $f:\inx\to\rr^q$ \citep{Alvarez12:KVF}.
Let $\hbspf$ be the RKHS of real-valued functions on $\inx$ with reproducing kernel $k$ and $\hbspf^q$ the product RKHS of functions $f := (f_1,\ldots,f_q)$ where $f_i\in\hbspf$ for all $i$ with an inner product $\langle f,g \rangle_{\hbspf^q} = \sum_{i=1}^q \langle f_i,g_i\rangle_{\hbspf}$ and norm $\|f\|_{\hbspf^q} = \sqrt{\sum_{i=1}^q\|f_i\|^2_{\hbspf}}$. 
For $\theta\in\Theta$, we define an operator $M_\theta$ on $\hbspf^q$ as
\begin{align*}
    M_\theta f := \ep[\bpsi(Z;\theta)^\top f(X)]
% &\stackrel{(a)}{=} \sum_{i=1}^q \ep[\psi_i(Z;\theta)f_i(X)],
    = \sum_{i=1}^q \ep[\psi_i(Z;\theta)f_i(X)],
%    &= \sum_{i=1}^p \eta_i(z;\theta)\langle f_i, k(x,\cdot)\rangle \\
%    &= \sum_{i=1}^p \langle f_i, \eta_i(z;\theta)k(x,\cdot)\rangle
\end{align*}
where $\psi_i$ denotes the $i$-th component of $\bpsi$. 
This operator takes an instrument $f\in\hbspf^q$ as input and returns the corresponding conditional moment restrictions.

The following lemma shows that $M_\theta$ satisfies the property of the original conditional moment restrictions.
\begin{lemma}\label{lem:optimal}
    For all $f\in\hbspf^q$, $M_{\theta_0} f = 0$.
\end{lemma}
\noindent Moreover, by Assumption \ref{asmp:a2} and \ref{asmp:a4},  $|M_\theta f| \leq \sum_{i=1}^q\|f_i\|_{\hbspf_i}\sqrt{\ep[\psi_i(Z;\theta)\psi_i(Z';\theta)k(X,X')]} < \infty$ where $(X',Z')$ is an independent copy of $(X,Z)$. 
Hence, $M_\theta$ is a bounded linear operator.
By Riesz's representation theorem, there exists a unique element $\mut_{\theta}$ in $\hbspf^q$ such that $M_\theta f = \langle f,\mut_\theta\rangle_{\hbspf^q}$ for all $f\in\hbspf^q$. 
Indeed, by the reproducing property,
\begin{equation*}
    M_\theta f = \sum_{i=1}^q\left\langle f_i,\ep[\xi^i_\theta (X,Z)]\right\rangle_{\hbspf_i} = \left\langle f,\ep[\bm{\xi}_\theta (X,Z)]\right\rangle_{\hbspf^q},
\end{equation*}
where $\bm{\xi}_{\theta}(x,z) := \left(
\psi_1(z;\theta)k(x,\cdot),\ldots,\psi_q(z;\theta)k(x,\cdot) \right)$ is the feature map in $\hbspf^q$ and $\xi^i_\theta$ denotes the $i$-th element of $\bm{\xi}_\theta$.
The equalities above are well-defined since $\bm{\xi}_\theta(x,z)$ is Bochner integrable \citep[Def. A.5.20]{SteChr2008}, \ie, $\ep\|\bm{\xi}_\theta(X,Z)\|_{\hbspf^p} \leq \sqrt{\ep\|\bm{\xi}_\theta(X,Z)\|_{\hbspf^p}^2} = \sqrt{\ep[ \bpsi(Z;\theta)^\top\bpsi(Z;\theta)k(X,X)]} < \infty$.

In other words, $\mut_\theta := \ep[\bm{\xi}_\theta(X,Z)]$ is a \emph{representer} of $M_\theta$ in $\hbspf^q$.
We define $\mut_\theta$ as \emph{conditional moment embedding} (CMME) of  $\ep[\bpsi(Z;\theta)|X]$ in $\hbspf^q$ relative to $P_X$.

%%% conditional moment embedding
\begin{definition}\label{def:cmme}
    For each $\theta\in\Theta$, let $\bm{\xi}_\theta(x,z) := \left(
\psi_1(z;\theta)k(x,\cdot), \ldots,\psi_q(z;\theta)k(x,\cdot) \right) \in \hbspf^q$. 
The \emph{conditional moment embedding} (CMME) is defined as 
    \begin{equation}\label{eq:embedding}
        \bm{\mu}_\theta := \int_{\inx\times\inz}\bm{\xi}_{\theta}(x,z) \; dP_{\mathit{XZ}}(x,z) \in \hbspf^q.
    \end{equation}
\end{definition}

4The CMME $\mut_\theta$ takes the form of a kernel mean embedding of $P_{\mathit{XZ}}$ with $\bxi_\theta$ as the feature map \citep{Muandet17:KME}. 
This is illustrated in Figure \ref{fig:cmme}.
Hence, given an i.i.d. sample $(x_i,z_i)_{i=1}^n$ from $P_{\mathit{XZ}}$, we can estimate $\mut_\theta$ simply by $\muth_\theta := \frac{1}{n}\sum_{i=1}^n\bm{\xi}_\theta (x_i,z_i)$. 
The following theorem establishes the $\sqrt{n}$-consistency of this estimator.

%%% consistency of CMME
\begin{theorem}\label{thm:cmme-consistency}
    Let $\sigma^2_{\theta} := \ep\|\bxi_\theta(X,Z)\|_{\hbspf^q}^2$ and
    assume that $\|\bxi_\theta(X,Z)\|_{\hbspf^q} < C_\theta < \infty$ almost surely. 
    Then, for any $0 < \delta < 1$, with probability at least $1-\delta$, 
    \begin{equation}
    \left\|\muth_\theta - \mut_\theta\right\|_{\hbspf^p} \leq 
    \frac{2C_\theta\log\frac{2}{\delta}}{n} + \sqrt{\frac{2\sigma^2_{\theta}\log\frac{2}{\delta}}{n}}.
    \end{equation}
\end{theorem}

Remarkably, $\muth_{\theta}$ converges at a rate $O_p(n^{-1/2})$ that is independent of the dimension of $(X,Z)$ and the RKHS $\hbspf^q$.
This is an appealing property because estimation and inference based on $\muth_\theta$ become less susceptible to the \emph{curse of dimensionality} (see, \eg, \citet{Khosravi19:NPIA} and references therein for the discussion).
Under certain assumptions, \citet{TolstikhinSM17:Minimax} established the minimax optimal rate for the kernel mean estimators like $\muth_\theta$.

The next theorem shows that $\mut_\theta$ provides a \emph{unique} representation of the CMR $\CMR(X,\theta)$ in $\hbspf^q$ relative to $P_X$.

%%% identification
\begin{theorem}\label{thm:identifiability}
    Assume that the kernel $k$ is ISPD. Then, for any $\theta_1,\theta_2\in\Theta$, $\CMR(x;\theta_1) = \CMR(x;\theta_2)$ for $P_X$-almost all $x$ if and only if $\mut_{\theta_1} = \mut_{\theta_2}$. 
\end{theorem}

To better understand Theorem \ref{thm:identifiability}, consider when $q=1$ and $k(x,x')=\varphi(x-x')$ is a shift-invariant kernel.
First, we have 
$\mut_\theta(\cdot) = \ep_X[\ep_Z[\bpsi(Z;\theta)k(X,\cdot)|X]] 
 = \ep_X[\ep_Z[\bpsi(Z;\theta)|X]k(X,\cdot)] 
 = \ep_X[\CMR(X;\theta)k(X,\cdot)]$.
It is then easy to show using Theorem \ref{thm:bochner} that $\mut_\theta(\cdot) = \int_{\rr^d} \phi(\omega;\theta) c(\omega,\cdot)\dd\Lambda(\omega)$ where $c(\omega,y) = \exp(i\omega^\top y) \neq 0$ and $\phi(\omega;\theta) := \ep_X[\CMR(X;\theta)\exp({i\omega^\top X})]$ is the Fourier transform (or characteristic function) of the Borel measurable function $\CMR(x;\theta)$ relative to $P_X$. 
Hence, if $\text{supp}(\Lambda)=\rr^d$, the uniqueness of $\mut_\theta$ follows from the uniqueness of $\phi(\omega;\theta)$.
\citet{Bierens82:ConsistentCM} was the first to observe the characterization of the CMR in terms of the integral transform and then used it to construct the consistent CM tests of functional form (cf. Section \ref{sec:related-work}).

Theorem \ref{thm:identifiability} shows that $\mut_{\theta}$ captures all information about $\ep[\bpsi(Z;\theta)|x]$ for every $x\in\inx$ for which $P_X(x) > 0$. 
Consequently, estimation and inference on CMR can be performed by means of $\mut_\theta$ using the existing kernel arsenal.
As mentioned earlier, for each $f\in\hbspf^q$ and $\theta\in\Theta$, the inner product $\langle f,\mut_\theta\rangle_{\hbspf^q} = \langle f,\ep[\bm{\xi}_\theta(X,Z)]\rangle_{\hbspf^q}$ can be interpreted as a restriction of conditional moments with respect to $f$.
Moreover, the investigator can inspect $\mut_{\theta}(x,z)$, which measures to what extent the moment conditions are violated at $(x,z)$, \ie, structural instability, in order to understand the nature of misspecification.

%%%%
\subsection{MAXIMUM MOMENT RESTRICTION WITH REPRODUCING KERNELS}
\label{sec:kmmr}

Based on the CMME $\mut_\theta$, we can now define the MMR as
\begin{equation}\label{eq:MMR}
    \mathbb{M}(\theta) := \sup_{\|f\|_{\hbspf^q}\leq 1} M_{\theta}f
        = \sup_{\|f\|_{\hbspf^q}\leq 1} \langle f,\mut_\theta\rangle_{\hbspf^q}
        = \|\bm{\mu}_\theta\|_{\hbspf^q}.
\end{equation}
By Theorem \ref{thm:identifiability}, $\mathbb{M}(\theta) \geq 0$ and $\mathbb{M}(\theta) = 0$ if and only if $\theta=\theta_0$. 
Put differently, $\mathbb{M}(\theta)$ measures how much the models associated with $\theta$ violate the original CMR in \eqref{eq:cond-moments}.

To obtain an expression for $\mathbb{M}(\theta)$, we define a real-valued kernel $h_\theta:(\inx\times\inz)\times(\inx\times\inz)\to\rr$ based on the feature map $\bxi_\theta:\inx\times\inz\to\hbspf^q$ as follows: 
\begin{align}\label{eq:kernel_h}
h_{\theta}((x,z),(x',z')) &:= \langle \bxi_{\theta}(x,z),\bxi_{\theta}(x',z') \rangle_{\hbspf^q} \nonumber \\ 
&= \bpsi(z;\theta)^\top\bpsi(z';\theta)k(x,x').
\end{align}
Then, a closed-form expression for $\mathbb{M}(\theta)$ in terms of the kernel $h_\theta$ follows straightforwardly.

%%% closed-form 
\begin{theorem}\label{thm:close-form}
    Assume that $\ep[h_\theta((X,Z),(X,Z))] < \infty$. 
    Then, $\mathbb{M}^2(\theta) = \ep[h_\theta((X,Z),(X',Z'))]$ where $(X',Z')$ is independent copy of $(X,Z)$ with the same distribution.
\end{theorem}

Finally, Mercer's representation \eqref{eq:mercer} of $k$ allows us to interpret $h_\theta$ and $\MM(\theta)$ in terms of a continuum of unconditional moment restrictions.

%%%% infinite moment restriction
\begin{theorem}\label{thm:expansion}
    Let $\{(\lambda_j,e_j)\}$ be eigenvalue/eigenfunction pairs associated with the kernel $k$ and $\bm{\zeta}_\theta^j(x,z) := \left(
    \psi_1(z;\theta)e_j(x), \ldots,\psi_q(z;\theta)e_j(x) \right)$.
    Then, for each $\theta\in\Theta$, $h_\theta((x,z),(x',z')) = \sum_j\lambda_j\bm{\zeta}_\theta^j(x,z)^\top \bm{\zeta}_\theta^j(x',z')$
    and $\mathbb{M}^2(\theta) = \sum_j\lambda_j\|\ep[\bm{\zeta}_\theta^j(X,Z)]\|^2_2$.
\end{theorem}

That is, we can interpret $\ep[\bm{\zeta}_\theta^j(X,Z)]$ as the UMR with $e_j$ acting as an instrument. 
Moreover, $\MM^2(\theta)$ can be viewed as a weighted sum of moment restrictions based on the sequence of weights and instruments $(\lambda_j,e_j)_j$.
As a result, the CM test based on $\MM^2(\theta)$ as a test statistic examines an infinite number of moment restrictions.
Note that $(\lambda_j,e_j)_j$ are defined implicitly by the choice of $k$.

%%%
\section{KERNEL CONDITIONAL MOMENT TEST WITH BOOTSTRAPPING}
\label{sec:kcm}

By virtue of Theorem \ref{thm:identifiability}, we can reformulate the CM testing problem \eqref{eq:cm-test} in terms of the MMR as
\begin{equation*}
    H_0: \MM^2(\theta) = 0, \quad H_1: \MM^2(\theta) \neq 0.
\end{equation*}
Given an i.i.d. sample $\{(x_i,z_i)\}_{i=1}^n$ from the distribution $P_{\mathit{XZ}}$, we consider the test statistic
\begin{equation}\label{eq:estimator}
    \MH_n^2(\theta) = \frac{1}{n(n-1)}\sum_{1\leq i\neq j \leq n} h_\theta((x_i,z_i),(x_j,z_j)),
\end{equation}
which is in the form of $U$-statistics \citep[Section
5]{serfling1980approximation}. 
Although there exist several potential estimators for $\MM^2(\theta)$, we focus on \eqref{eq:estimator} as it is a minimum-variance unbiased estimator with appealing asymptotic properties.
Moreover, \eqref{eq:estimator} also provides a basis for the estimation of $\theta_0$ simply by minimizing $\MH_n^2(\theta)$ with respect to $\theta\in\Theta$. 
Preliminary results on estimation are given in Appendix \ref{sec:estimation}.

Next, we characterize the asymptotic distributions of $\MH^2_n(\theta)$ under the null and alternative hypotheses.

%%%% consistency and asymptotic normality for M
\begin{theorem}
    \label{thm:asymptotic}
    Assume that $\mathbb{E}[h^2_\theta((X,Z),(X',Z'))] < \infty$ for all $\theta \in \Theta$.
    Let $U:=(X,Z)$ and $U':=(X',Z')$. Then, the following statements hold.
    
    \begin{enumerate}[label=\textnormal{(\arabic*)}]
        \item If $\theta \neq \theta_0$, $\MH^2_n(\theta)$ is asymptotically normal with
            \begin{equation*}
            \sqrt{n}\left(\MH_n^2(\theta)-\MM^2(\theta)\right)
            \stackrel{d}{\to} \mathcal{N}(0,4 \sigma_{\theta}^2),
            \end{equation*}
        where $\sigma_{\theta}^2 = \mathrm{Var}_U\left[\ep_{U'}[h_{\theta}(U,U')] \right]$.
        \item If  $\theta = \theta_0$, then $\sigma_{\theta}^2 = 0$ and
            \begin{equation}\label{eq:alt-asymp}
            n \MH_n^2(\theta) \stackrel{d}{\to} \sum_{j=1}^\infty \tau_j
            \left(W^2_j - 1 \right), 
            \end{equation}
        where $W_j\sim\mathcal{N}(0,1)$ and $\{\tau_j\}$ are the eigenvalues of $h_{\theta}(u,u')$, \ie, they are the solutions of
        $\tau_j \phi_j(u) = \int h_{\theta}(u,u')\phi_j(u') dP(u')$ for non-zero $\phi_j$.
    \end{enumerate}
\end{theorem}

%\begin{itemize}
%     \item[a)]  By assumption $\text{Var}_{u,u'\sim \mathbb{P}}[h(u,u')] > 0$ and $\mathbb{M}^2(\theta) \neq 0$. Thus $\sigma_h^2 \neq 0$ and the result follows from section 5.5.1 in \cite{serfling1980approximation}.
% \item[b)] By \eqref{eq:MMR}, $\mathbb{M}^2(\theta) = 0$ implies $\bm{\mu}_\theta = 0$. By the definition of $h$, see \eqref{eq:kernel_h}, we thus get $\ep_{u'\sim \mathbb{P}}[h(u,u')]$ = 0 and hence $\sigma_h^2 =0$. The result then follows from section 5.5.2 in \cite{serfling1980approximation}.
% \end{itemize}

% {
% \remark{
% If the assumption $\text{Var}_{u,u'\sim \mathbb{P}}[h(u,u')] > 0$ in theorem \ref{thm:asymptotic} does not hold, i.e., $\text{Var}_{u,u'\sim \mathbb{P}}[h(u,u')] = 0$, then $\mathbb{M}_n^2(\theta) = \mathbb{M}^2(\theta)$, trivially.
% }}

As we can see, $n\MH_n^2(\theta) < \infty$ with probability one under the null $\theta=\theta_0$ and diverts to infinity at a rate $\mathcal{O}(\sqrt{n})$ under any fixed alternative $\theta\neq \theta_0$. 
Hence, a consistent CM test can be constructed as follows: if $\gamma_{1-\alpha}$ is the $1-\alpha$ quantile of the CDF of $n\MH^2_n(\theta)$ under the null $\theta=\theta_0$, we reject the null with significance level $\alpha$ if $n\MH^2_n(\theta) \geq \gamma_{1-\alpha}$.

%% consistency of the test
\begin{proposition}[\citet{Arcones92:Bootstrap}; p.~671]
\label{prop:test-consistency}
Assume the conditions of Theorem \ref{thm:asymptotic}. 
The test that rejects the null $\theta=\theta_0$ when $n\MH_n^2(\theta) > \gamma_{1-\alpha}$ is consistent against any fixed alternative $\theta\neq\theta_0$, \ie, the limiting power of the test is one.
\end{proposition}

%%%%
\begin{algorithm}[t!]
    \caption{KCM Test with bootstrapping}
    \label{alg:kcm}
    \begin{algorithmic}
    \REQUIRE{Bootstrap sample size $B$, significance level $\alpha$}
    % \ENSURE{aa}
    \FOR{$t \in \{1, \ldots, B\}$}
        \STATE Draw $(w_1, \ldots, w_n) \sim \mathrm{Mult}(n; \frac{1}{n}, \ldots, \frac{1}{n})$
        \STATE $\rho_i \leftarrow (w_i-1)/n$ for $i=1,\ldots,n$ 
        \STATE $\MH_n^*(\theta) \leftarrow  \sum_{i\neq j} \rho_i\rho_j h_\theta((x_i,z_i),(x_j,z_j))$
        \STATE $a_t \leftarrow n\MH_n^*(\theta) $
    \ENDFOR
    \STATE $\hat{\gamma}_{1-\alpha} :=$ empirical $(1-\alpha)$-quantile of $\{ a_t\}_{t=1}^B$
    % \STATE $\text{p-value} := \frac{1}{B} \sum_{t=1}^B \mathbbm{1}[\MH_n^2(\theta) > a_t]$ 
    \STATE Reject $H_0$ if $\hat{\gamma}_{1-\alpha} < n \MH_n^2(\theta)$ (see \eqref{eq:estimator})
    % \IF{$\hat{\gamma}_{1-\alpha} < \MH_n^*(\theta)$}
    %     \STATE Reject $H_0$
    % \ENDIF
    \end{algorithmic}
\end{algorithm}

Unfortunately, the limiting distribution in \eqref{eq:alt-asymp} and its $1-\alpha$ quantile do not have an analytic form.
Following recent work on kernel-based tests \citep{Liu16:KSD,Chwialkowski16:KGFT,gretton2012kernel}, we propose to approximate the critical values using the bootstrap method proposed by \citet{Arcones92:Bootstrap,Huskova93:BSConsist}, which was previously used in \citet{Liu16:KSD}.
Specifically, we first draw multinomial random weights $(w_1,\ldots,w_n)\sim\text{Mult}(n; \frac{1}{n},\ldots,\frac{1}{n})$ and compute the bootstrap sample 
$\MH_n^*(\theta) = (1/n^2)\sum_{1 \leq i\neq j \leq n}(w_i - 1)(w_j - 1)h_\theta((x_i,z_i),(x_j,z_j))$.
We then calculate the empirical quantile $\hat{\gamma}_{1-\alpha}$ of $n\MH^*_n(\theta)$.
% The consistency of $\hat{\gamma}_{1-\alpha}$ for degenerate $U$-statistics was established \citep{Arcones92:Bootstrap,Huskova93:BSConsist}.
For degenerate $U$-statistics, $\hat{\gamma}_{1-\alpha}$ is a consistent estimate of $\gamma_{1-\alpha}$ \citep{Arcones92:Bootstrap,Huskova93:BSConsist}.

We summarize our bootstrap kernel conditional moment (KCM) test in Algorithm 1. Note that the proposed test checks the CMR for a \emph{given} parameter $\theta$ and does not take into account the estimation error of $\theta$. We defer a full treatment of interplay between parameter estimation and hypothesis testing to future work.

% We summarize our bootstrap kernel conditional moment (KCM) test in Algorithm \ref{alg:kcm}.
% The KCM test checks the CMR for a \emph{given} parameter $\theta$. However, to test the correct specification of the model, we need to test the existence of $\theta_0 \in \Theta$ for which the CMR hold. This additionally requires the estimation of the parameters. We further need to derive the distribution of $\MH_n^2(\hat{\theta})$, taking into account that $\hat{\theta}$ is random and depends on data as well.
% While we give preliminary results on the estimation, we defer a full treatment of the misspecification test to future work.

%%%%
\section{RELATED WORK}
\label{sec:related-work}

Existing CM tests can generally be categorized into two classes.
The former is based on a transformation of CMR into a continuum of unconditional counterparts, \eg, \citet{Bierens82:ConsistentCM,Bierens90:FuncForm}, \citet{Jong96:Dependence}, \citet{Bierens97:Asymptotic}, and \citet{Donald03:Likelihood} to name a few.
The latter employs nonparametric kernel estimation which includes \citet{Zheng96:NonparamTest,LiWang98:NonparamTest,Fan00:KernelICM} among others.
While both classes lead to consistent tests, they exhibit different asymptotic behaviors; see, \eg, \citet{Fan00:KernelICM,Delgado06:Consistent} for detailed comparisons.

%%%%
\vspace{-5pt}
\paragraph{A continuum of unconditional moments.}

One of the classical approaches is to find a parametric weighting function $w(x,\eta)$ such that
\begin{equation*}
    \ep[\bpsi(Z;\theta)|X] = \mathbf{0} \text{ a.s.} \;
    \Leftrightarrow \;
    \ep[\bpsi(Z;\theta)w(X,\eta)] = \mathbf{0},
\end{equation*}
\noindent for almost all $\eta\in\Xi\subseteq\rr^m$ where $\eta$ is a nuisance parameter.
\citet{Newey85:Specification} and \citet{Tauchen85:CMT} proposed the so-called M-test using a finite number of weighting functions.
Since it imposes only a finite number of moment conditions, the test cannot be consistent against all possible alternatives and power against specific alternatives depends on the choice of these weighting functions.
\citet{Jong96:Dependence} and \citet{Donald03:Likelihood} showed that this issue can be circumvented by allowing the number of moment conditions to grow with sample size.
Although our KCM test generally relies on infinitely many moment conditions, one can impose finitely many conditions using the finite dimensional RKHS such as those endowed with linear and polynomial kernels or resorting to finite-dimensional kernel approximations.

\citet{Stinchcombe98:Nuisance} showed that there exists a wide range of $w(x,\eta)$ that lead to consistent CM tests.
They call these functions ``totally revealing''.
For instance, \citet{Bierens82:ConsistentCM} proposed the first consistent specification test for nonlinear regression models using $w(x,\eta) = \exp(i\eta^\top x)$ for $\eta\in\rr^d$.
Similarly, \citet{Bierens90:FuncForm} used $w(x,\eta) = \exp(\eta^\top x)$ for $\eta\in\rr^d$.
An indicator function $w(x,\eta) = \mathbbm{1}(\alpha^\top x \leq \beta)$ with $\eta=(\alpha,\beta) \in \mathbb{S}^d\times (-\infty,\infty)$ where $\mathbb{S}^d=\{\alpha\in\rr^d : \|\alpha\|=1\}$ was used in \citet{Escanciano06:Projection} and \citet{Delgado06:Consistent}.
Other popular weighting functions include power series, Fourier series, splines, and orthogonal polynomials, for example.
In light of Theorem \ref{thm:expansion}, the KCM test falls into this category where weighting functions are eigenfunctions associated with the kernel $k$.
%For instance, the indicator kernel $k(x,x') = \mathbbm{1}[x=x']$ or RBF kernel with the diminishing bandwidth $\sigma \to 0$ lead to the MMR based on the indicator weighting functions.

Since $w(x,\eta)$ depends on the nuisance parameter $\eta$, \citet{Bierens82:ConsistentCM} suggested to integrate $\eta$ out, resulting in an \emph{integrated conditional moment} (ICM) test statistic:
\begin{equation}
    \label{eq:icm-test}
    \widehat{T}_n(\theta) = \int_{\Xi} \| \widehat{Z}_n(\eta)\|_2^2\dd\nu(\eta),
\end{equation}
where $\Xi$ is a compact subset of $\rr^d$, $\nu(\eta)$ is a probability measure on $\Xi$, and $\widehat{Z}_n(\eta) := (1/\sqrt{n})\sum_i\bpsi(z_i;\theta)w(x_i,\eta)$.
The limiting null distribution of the ICM test was proven to be a zero-mean Gaussian process \citep{Bierens90:FuncForm}.
\citet{Bierens97:Asymptotic} also characterizes the asymptotic null distribution of a general class of real-valued weighting functions.

The following theorem establishes the connection between the KCM and ICM test statistics.

%%%% connection to Bieren's tests
\begin{theorem}\label{thm:kcm-icm}
    Let $k(x,x') = \varphi(x-x')$ be a shift-invariant kernel on $\rr^d$. Then, we have
    $$\MM^2(\theta) = \frac{1}{(2\pi)^{d/2}}\int_{\rr^d} \left\|\ep[\bpsi(Z;\theta)\exp(i\omega^\top X)]\right\|_2^2 \dd\Lambda(\omega)$$
    where $\Lambda$ is a Fourier transform of $k$.
\end{theorem}

This theorem is quite insightful as it describes the KCM test statistic as the ICM test statistic $\widehat{T}_n(\theta)$ of \citet{Bierens82:ConsistentCM} where the distribution on the nuisance parameter $\omega$ is a Fourier transform of the kernel.
For instance, the Gaussian kernel $k(x,x')=\exp(-\|x-x'\|_2^2/2\sigma^2)$ corresponds to the Gaussian density $\Lambda(\omega) = \exp(-\sigma^2\|\omega\|_2^2/2)$; see \citet[Table 2.1]{Muandet17:KME} for more examples.
Note that both weighting functions and integrating measures are implicitly determined by the kernel $k$.
Unlike ICM tests, KCM tests can be evaluated without solving the high-dimensional numerical integration \eqref{eq:icm-test} explicitly.
Moreover, KCM tests can be easily generalized to $\inx$ that is not necessarily a subset of $\rr^d$. 

\citet{Carrasco00:Continuum} also considers a similar setting that involves a continuum of moment conditions in RKHS.
Their approach, however, differs significantly from ours. 
First, they consider a specific case where the Hilbert space is a set of square integrable functions of a scalar $t\in[0,T]$ with the unconditional moment conditions $\ep[\bpsi_t(X,\theta_0)] = \mathbf{0}$ for all $t\in [0,T]$.
Second, their key question is to identify the optimal choice of weighting matrix in GMM. 
Third, estimation is actually based on a truncation of infinite moment conditions.
Lastly, they also proposed the CM test similar to the ICM tests, but it can handle only the case with $Z\in\rr$, while our test is applicable to any domain with a valid kernel.

%%%%
\vspace{-5pt}
\paragraph{Nonparametric kernel estimation.}

The second class of tests, known as \emph{smooth tests} \citep{Zheng96:NonparamTest,LiWang98:NonparamTest,Fan00:KernelICM}, adopts the statistic of the form
\begin{equation}\label{eq:kde-test}
    T(\theta) = \ep[\bpsi(Z;\theta)^\top\ep[\bpsi(Z;\theta)|X]f(X)].
\end{equation} 
Based on the kernel estimator of $\ep[\bpsi(Z;\theta)|X]f(X)$, the empirical estimate of \eqref{eq:kde-test} can be expressed as
\begin{equation}\label{eq:kde-emp}
    \widehat{T}_n(\theta) = \frac{1}{n(n-1)h^d}\sum_{1\leq i\neq j \leq n}\bpsi(z_i;\theta)^\top\bpsi(z_j;\theta)K_{ij}
\end{equation}
\noindent where $K_{ij} = K((x_i - x_j)/h)$, $K(\cdot):\rr^d\to\rr$ is a normalized kernel function and $h$ is a smoothing parameter.
Here, we emphasize that existing smooth tests rely on the kernel density estimator (KDE) in which the kernel used is not necessarily a reproducing kernel.   
For the smooth test to be consistent, $h$ must vanish as $n\to\infty$, whereas our KCM test is consistent even when the kernel is fixed.
Nevertheless, if $K(\cdot)$ is a reproducing kernel, the test statistic $\widehat{T}_n(\theta)$ with a fixed smoothing parameter $h$ resembles the KCM test statistic \eqref{eq:estimator}.
In fact, \citet{Fan00:KernelICM} has shown that the ICM test is a special case of the kernel-based test with a fixed smoothing parameter.
However, the critical drawback of the nonparametric kernel-based tests is that they have non-trivial power only against local alternatives that approach the null at a slower rate than $1/\sqrt{n}$, due to the slower rate of convergence of kernel density estimators, \ie, $O((nh^{d/2})^{-1/2})$ as $h\to 0$ \citep{Fan00:KernelICM}.
Moreover, these tests are susceptible to the curse of dimensionality.

Last but not least, the kernel estimator is also a key ingredient in empirical likelihood-based CM tests  \citep{Tripathi03:CMRTest,Kitamura04:EL,Dominguez04:GMM}.

%%%%

\begin{figure*}[t!]
    \centering
    \begin{subfigure}[t]{0.33\textwidth}
        \centering
        \includegraphics[width=0.97\textwidth]{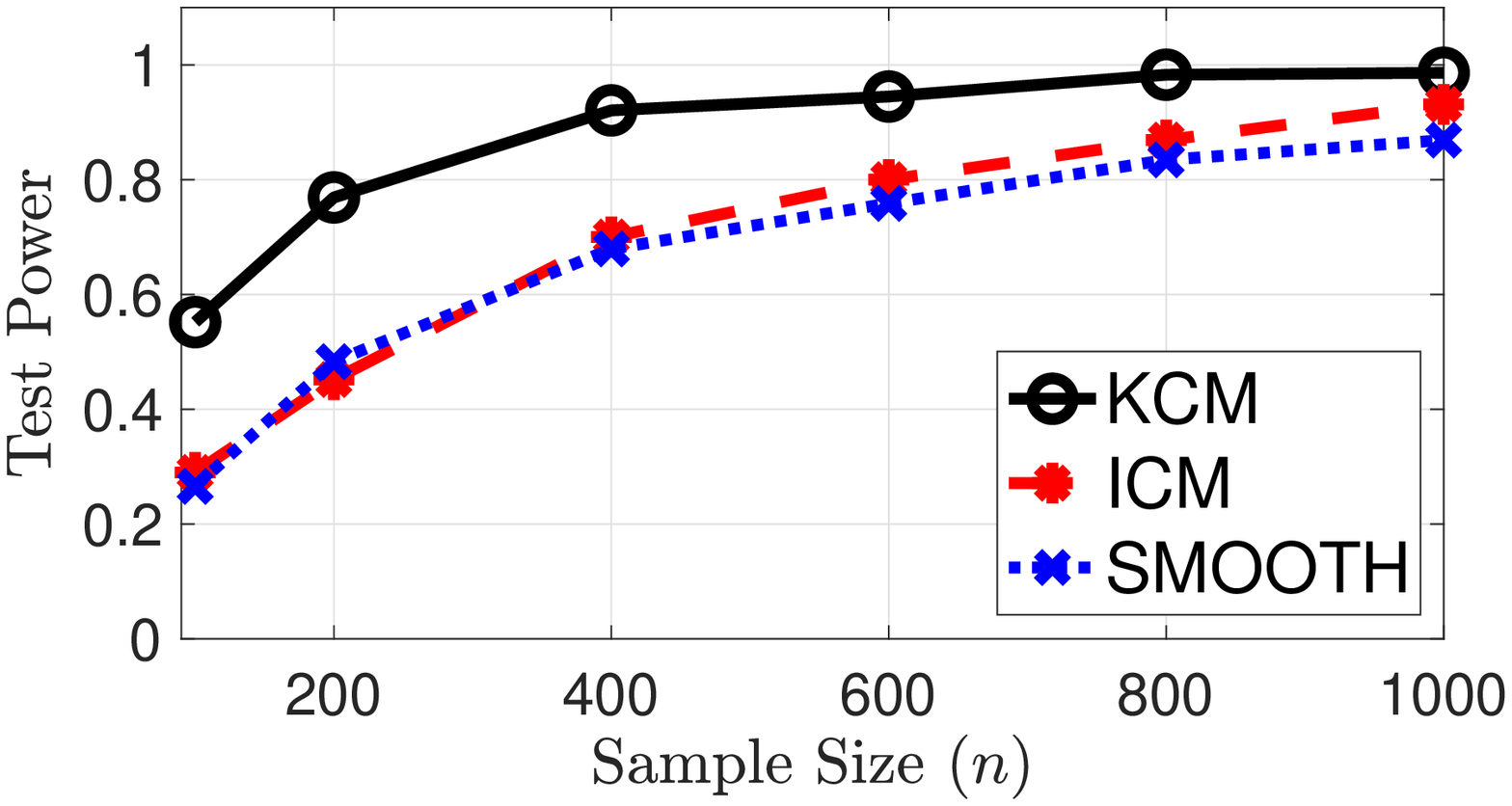}
        \subcaption{\texttt{REG-HOM} ($\delta=0.01$)}\label{fig:reg-hom-n}
    \end{subfigure}%
    \hfill
    \begin{subfigure}[t]{0.33\textwidth}
        \centering
        \includegraphics[width=0.97\textwidth]{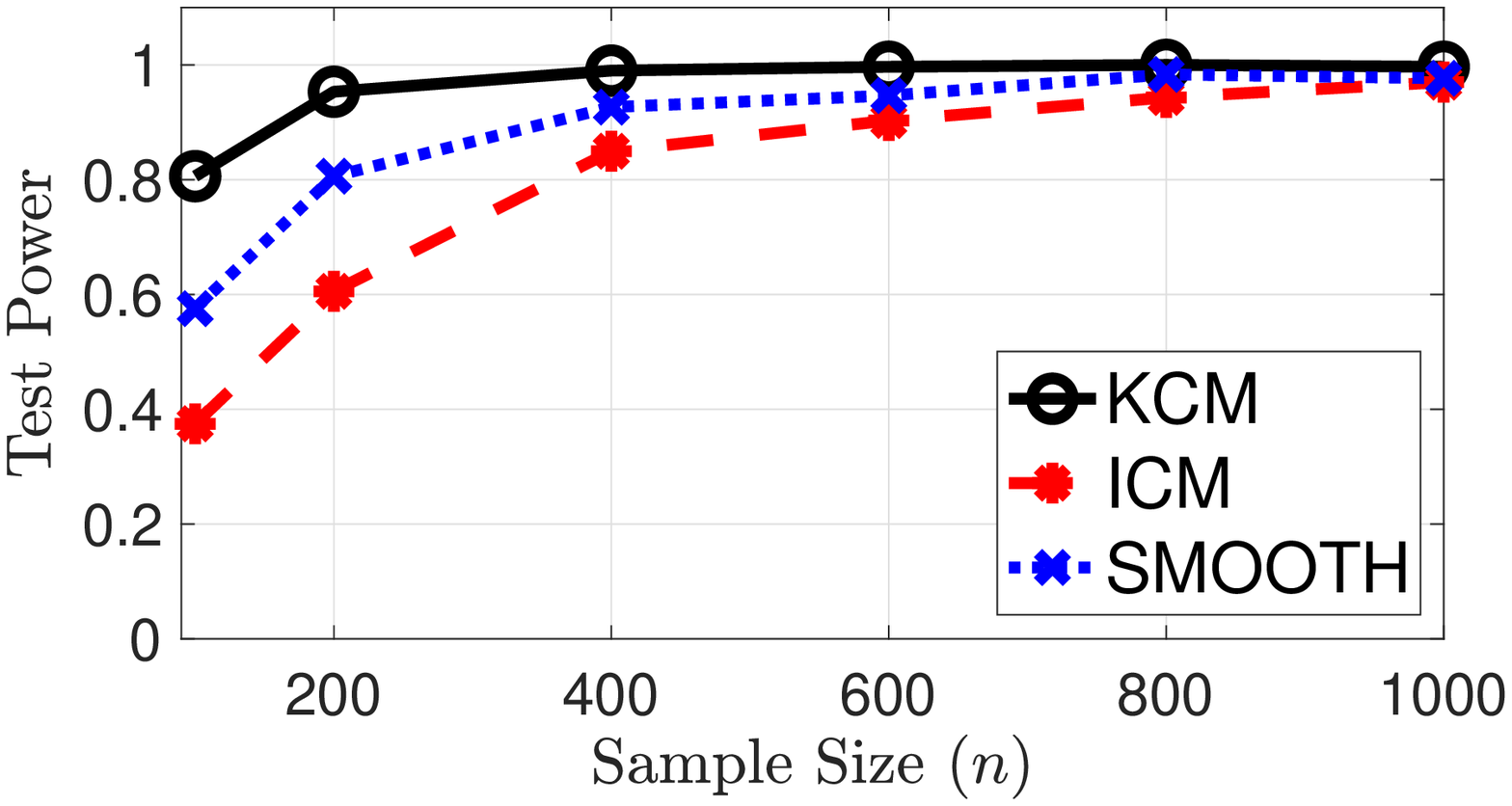}
        \subcaption{\texttt{REG-HET} ($\delta=0.01$)}\label{fig:reg-het-n}
    \end{subfigure}%
    \hfill
    \begin{subfigure}[t]{0.33\textwidth}
        \centering
        \includegraphics[width=0.97\textwidth]{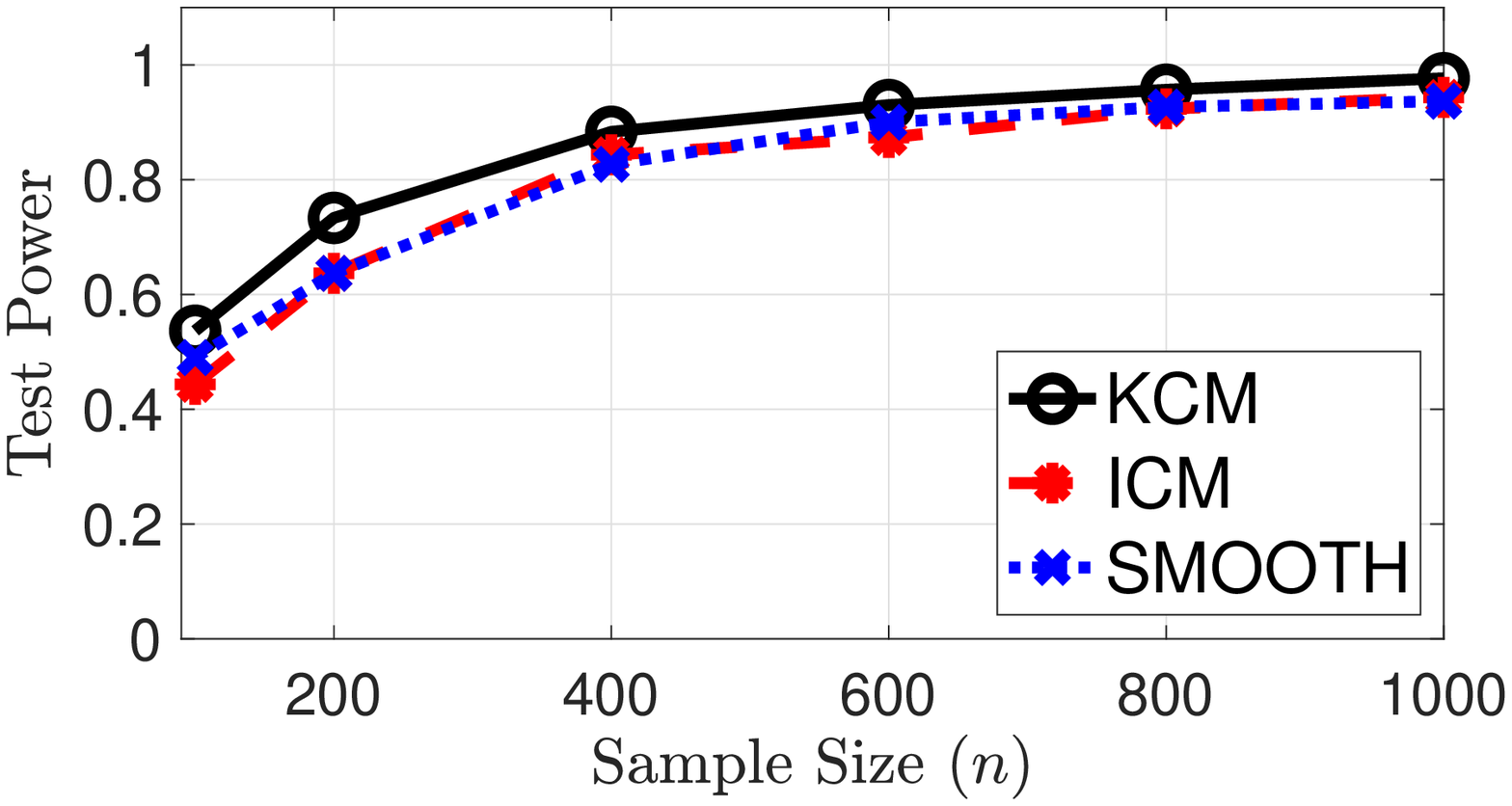}
        \subcaption{\texttt{SIMEQ} ($\delta=0.01$)}\label{fig:simeq-n}
    \end{subfigure}%%%
    \vspace{2.5em}
    \begin{subfigure}[t]{0.33\textwidth}
        \centering
        \includegraphics[width=0.97\textwidth]{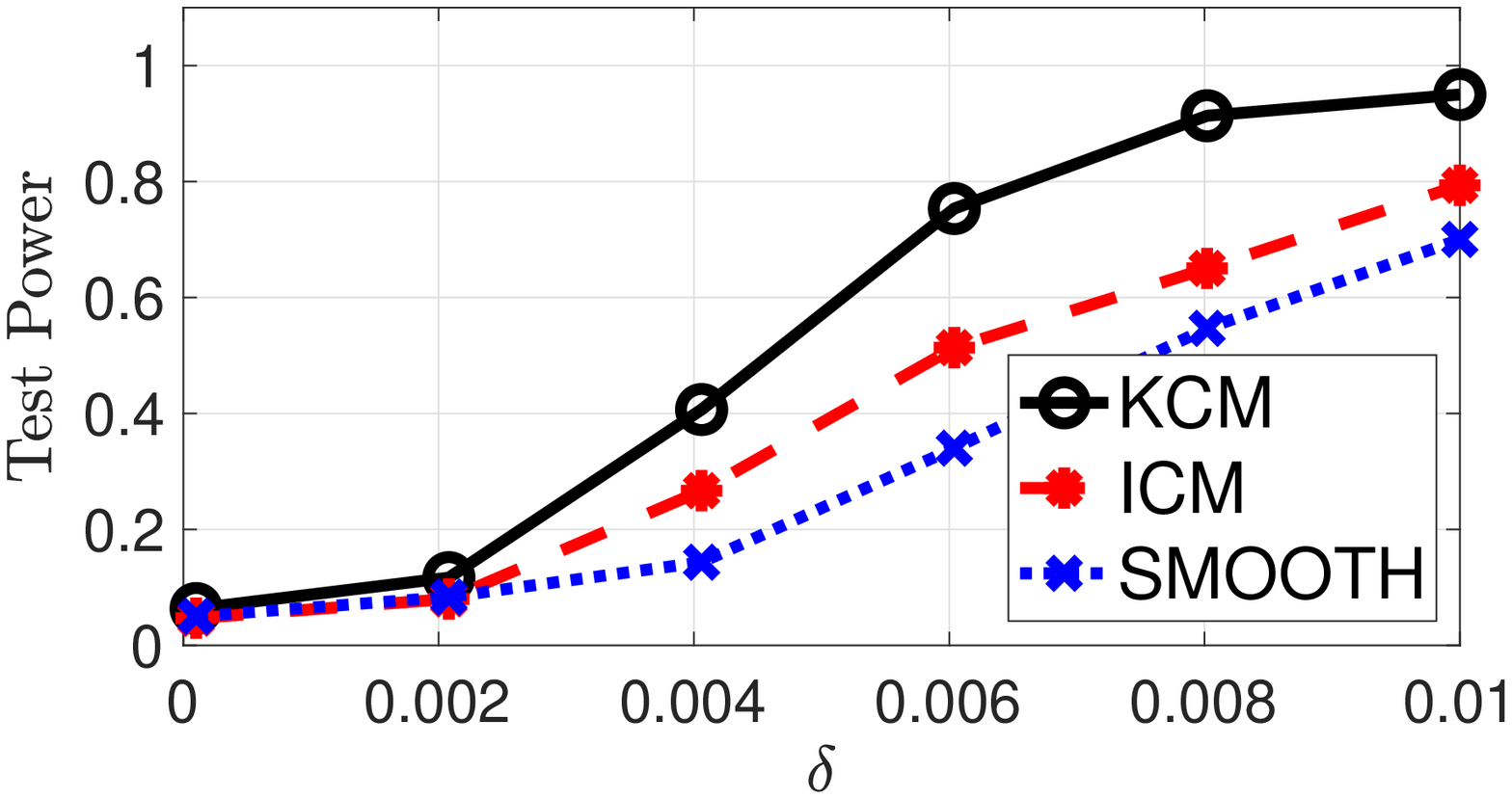}
        \subcaption{\texttt{REG-HOM} ($n=500$)}\label{fig:reg-hom-d}
    \end{subfigure}%
    \begin{subfigure}[t]{0.33\textwidth}
        \centering
        \includegraphics[width=0.97\textwidth]{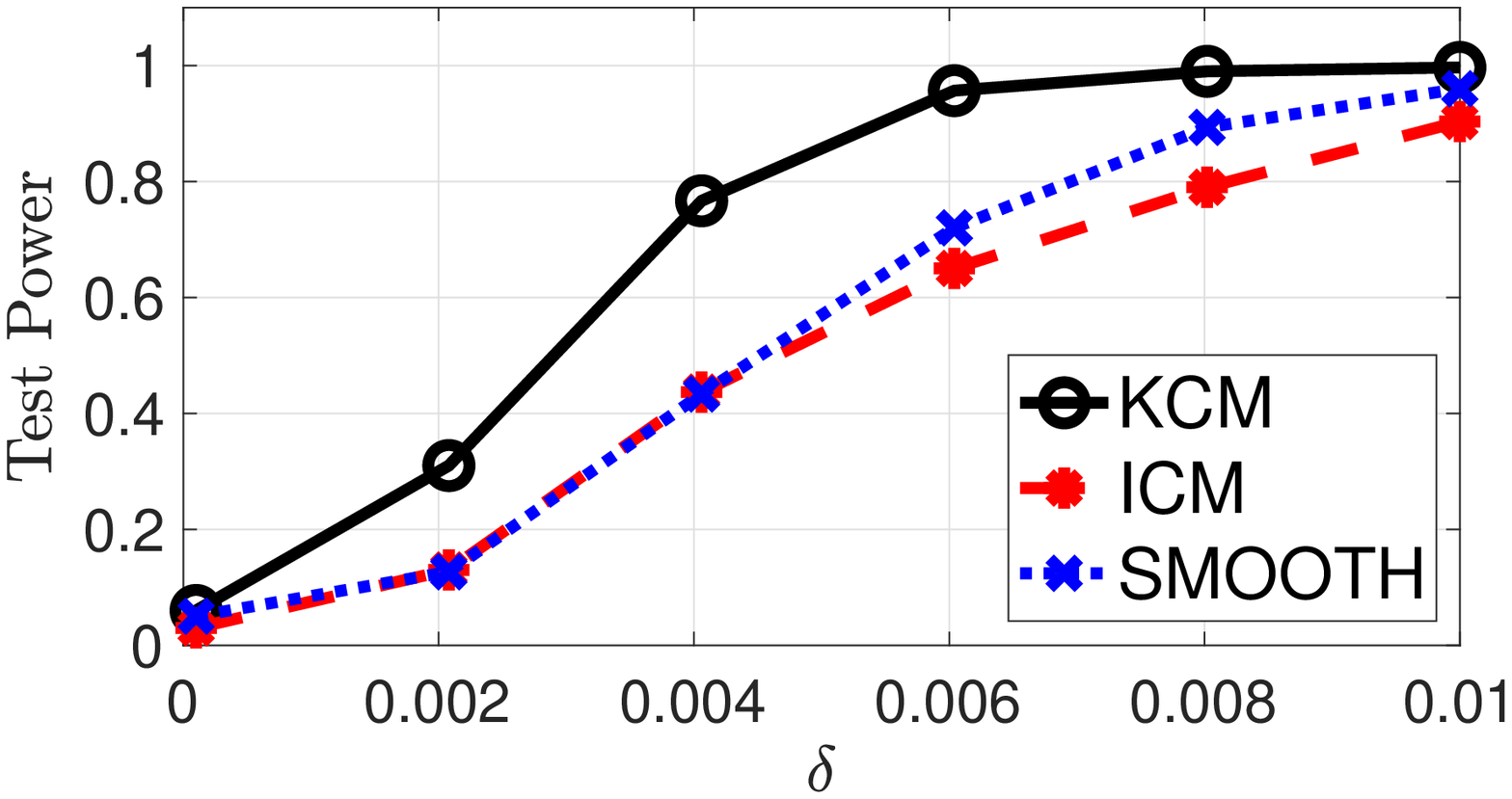}
        \subcaption{\texttt{REG-HET} ($n=500$)}\label{fig:reg-het-d}
    \end{subfigure}%
    \begin{subfigure}[t]{0.33\textwidth}
        \centering
        \includegraphics[width=0.97\textwidth]{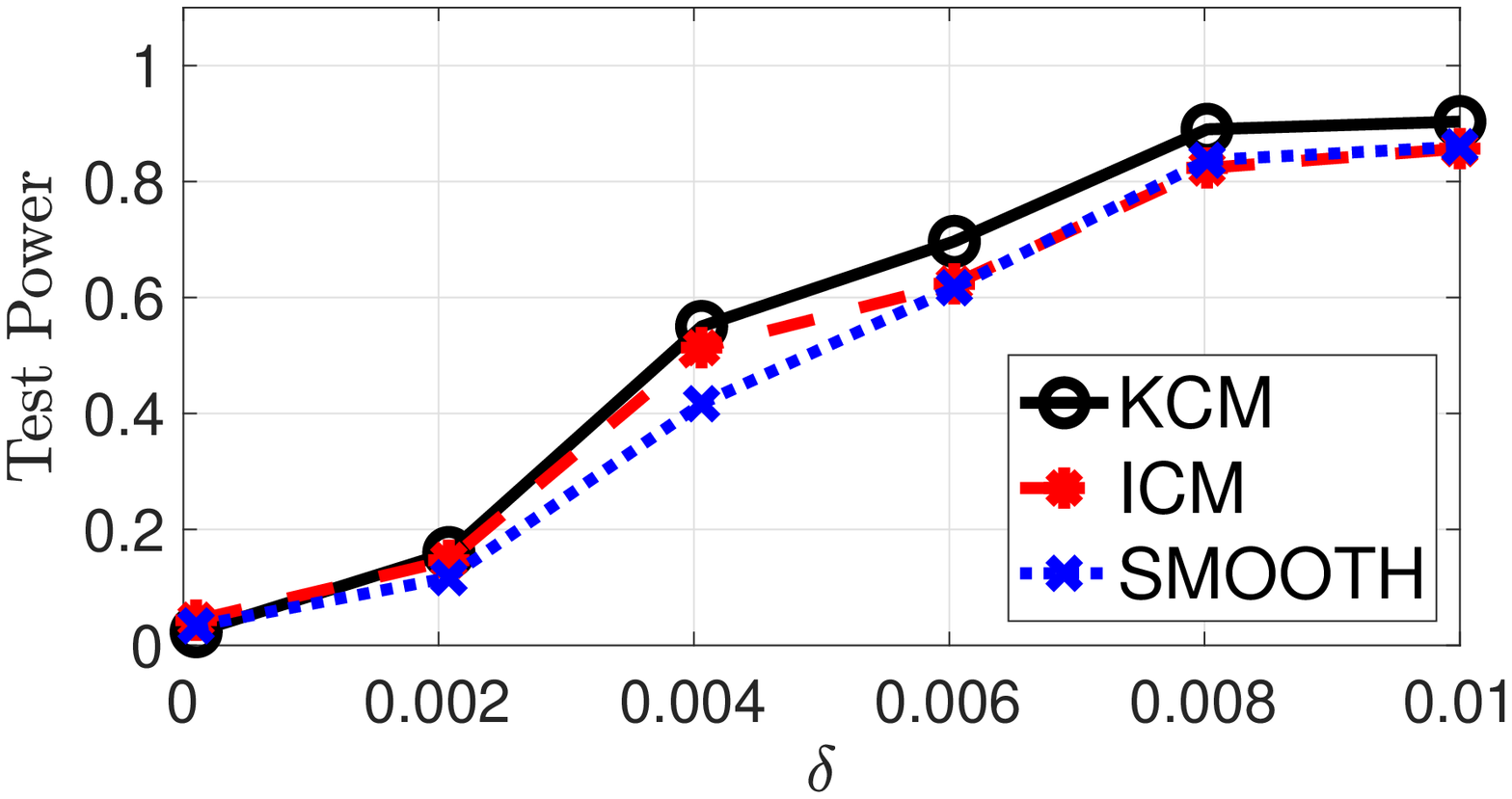}
        \subcaption{\texttt{SIMEQ} ($n=500$)}\label{fig:simeq-d}
    \end{subfigure}%
    \caption{The test powers of KCM, ICM, and smooth tests averaged over 300 trials as we vary the values of $n$ (top) and $\delta$ (bottom). Type-I errors of these tests are shown in Figure \ref{fig:results-typei} in Appendix \ref{sec:type1-error}. 
    See main text for the interpretation.}
    \label{fig:results}
\end{figure*}

%%%%

%%%%
\vspace{-5pt}
\paragraph{Kernelized Stein discrepancy (KSD).}

Stein's methods \citep{Stein72:Stein} are among the most popular techniques in statistics and machine learning.
One notable example is the Stein discrepancy which aims to characterize complex, high-dimensional distribution $p(x) = \tilde{p}(x)/N$ with intractable normalization constant $N = \int \tilde{p}(x)\dd x$ using a \emph{Stein operator} $\mathcal{A}_p$ such that
\begin{equation}\label{eq:stein-iden}
    p = q \quad \Leftrightarrow \quad \ep_{x\sim q}[\mathcal{A}_p f(x)] = 0, \;\; \forall f,
\end{equation}
\noindent where $\mathcal{A}_p f(x) := \nabla_x\log p(x)f(x) + \nabla_{x} f(x)$. 
Here, we assume for simplicity that $x\in\rr$.
The Stein operator $\mathcal{A}_p$ depends on the density $p$ through its \emph{score function} $s_p(x) := \nabla_x\log p(x) = \frac{\nabla_x p(x)}{p(x)}$, which is independent of $N$.
When $p\neq q$, the expectation in \eqref{eq:stein-iden} gives rise to a discrepancy 
\begin{eqnarray}\label{eq:stein-disc}
   \mathbb{S}_f(p,q) &:=& \ep_{x\sim q}[\mathcal{A}_p f(x)] \nonumber \\
   &=& \ep_{x\sim q}[(s_p(x) - s_q(x))f(x)].
\end{eqnarray}
See, also, \citet[Lemma 2.3]{Liu16:KSD}. The Stein discrepancy has led to numerous applications such as variance reduction \citep{Oates17:Control} and goodness-of-fit testing \citep{Liu16:KSD,Chwialkowski16:KGFT}, among others.

Like \eqref{eq:uncond-moments}, we can observe that \eqref{eq:stein-iden} is indeed a set of unconditional moment conditions.
To make an explicit connection between Stein discrepancy and CMR, we need to assume access to the probability densities.
Let $\mathcal{P}_{\Theta}$ be a space of probability densities $p(z; \theta)$ such that $\theta \mapsto p(z;\theta)$ is injective. 
We choose $\bpsi(z;\theta) = \nabla_z \log p(z;\theta) =: s_{\theta}(z)$ as the associated score function.\footnote{This differs from the standard definition of score function as $\nabla_\theta \log p(z|\theta)$ in the interpretation of maximum likelihood as generalized method of moments \citep{Hall05:GMM}.} 
This yields the following CMR:
\begin{equation}\label{eq:stein-cmr}
\ep[\nabla_z\log p(Z;\theta_0)\,|\,X] = \mathbf{0},\quad P_X\text{-a.s.}
\end{equation}
For any $\theta\in\Theta$, it follows that 
$\ep[\bpsi(Z;\theta)^\top f(X)] = \ep[s_{\theta}(Z)^\top f(X) - s_{\theta_0}(Z)^\top f(X)] = \ep[(s_{\theta}(Z) - s_{\theta_0}(Z))^\top f(X)] =: \Delta_f(\theta,\theta_0)$.
While $\Delta_f(\theta,\theta_0)$ resembles the Stein discrepancy in \eqref{eq:stein-disc}, we highlight the key differences.
First, this characterization requires that the model is correctly specified, \ie, $p(z;\theta_0)$ is observationally indistinguishable from the underlying data distribution.
Second, like the Stein discrepancy, it can be interpreted as the $f(x)$-weighted expectation of the score difference $s_{\theta} - s_{\theta_0}$. In contrast, the weighting function $f(x)$ in our setting depends only on $X$, which is a subvector of $Z$. 
We provide further discussion about this discrepancy measure in Appendix \ref{sec:cmd}.
The following theorem follows directly from the preceeding observation.
\begin{theorem}\label{thm:stein}
    Let $\mathcal{P}_{\Theta}$ be a space of probability densities $p(z; \theta)$. 
    Assume that $\theta \mapsto p(z;\theta)$ is injective and $\theta_0\in\Theta$. 
    If $\bpsi(z;\theta) = \nabla_z \log p(z;\theta)$ and $X=Z$, we have $\mathbb{S}_f(p(z;\theta),p(z;\theta_0)) = \Delta_f(\theta,\theta_0)$.
\end{theorem}

Mostly related to our work are the RKHS-based Stein's methods \citep{Liu16:KSD,Chwialkowski16:KGFT}.
Specifically, if we assume the conditions of Theorem \ref{thm:stein} and that $f$ belongs to the RKHS, it follows that $\Delta(\theta,\theta_0) := \sup_f \|\Delta_f(\theta,\theta_0)\|_2$ coincides with the kernelized Stein discrepancy (KSD) proposed in \citet{Liu16:KSD} and \citet{Chwialkowski16:KGFT}.
We will elaborate on this connection in further detail in future work.

%%%
\section{EXPERIMENTS} 
\label{sec:experiments}

We report the finite-sample performance of the KCM test against two well-known consistent CM tests, namely ICM test and smooth test, as discussed in Section \ref{sec:related-work}.
We evaluate all tests with a bootstrap size $B=1000$ and a significance level $\alpha=0.05$.

\begin{enumerate}[label=\textnormal{(\arabic*)},noitemsep]
    \item \texttt{\bfseries KCM}: The bootstrap KCM test using $U$-statistic in Algorithm \ref{alg:kcm}.
    We use the RBF kernel with bandwidth chosen by the median heuristic.
    
    \item \texttt{\bfseries ICM}: The test based on an integration over weighting functions. 
    Following \citet{Stute97:ModelCheck} and \citet{Delgado06:Consistent}, we use \eqref{eq:icm-test} as the test statistic with $w(x,\eta) = \mathbbm{1}(x \leq \eta) = \prod_{j=1}^d\mathbbm{1}(x_j\leq \eta_j)$ where $\mathbbm{1}(\cdot)$ is an indicator function.
    The density $\nu$ is chosen to be the empirical distribution of $X$.
    This leads to a simple test statistic $t_n = \sum_{i=1}^n r_n(x_i)^\top r_n(x_i)$ where $r_n(x) := \frac{1}{n}\sum_{i=1}^n\bpsi(z_i;\theta)\mathbbm{1}(x_i \leq x)$.
    We follow the bootstrap procedure in \citet[Sec. 4.3]{Delgado06:Consistent} to compute the critical values. 
    
    \item \texttt{\bfseries Smooth}: The test based on nonparametric kernel estimation. 
    We use \eqref{eq:kde-emp} as the test statistic. 
    The kernel is the standard Gaussian density function whose bandwidth is chosen by the rule-of-thumb $h=n^{-1/5}$.
    Note that the median heuristic is not applicable here because the bandwidth $h$ does not vanish, as required.
    The critical values are obtained using the same bootstrap procedure as in \citet[Sec. 4.2]{Delgado06:Consistent}. 
\end{enumerate}

%%%
\vspace{-5pt} 
\paragraph{Testing a regression function (\texttt{REG}).}
We follow a similar simulation of regression model used in \citet{Lavergne16:Hausman}.
% In this setting, the null hypothesis is 
% $$H_0: \ep[Y - \bm{\beta}_0^\top X\,|\, X] = 0 \quad \text{a.s. for some $\bm{\beta}_0\in\rr^{d+1}$}$$
% where $X \in \rr^d$ and $Y$ is a univariate random variable, \ie, $Z=(Y,X)$.
In this setting, for a given estimate $\hat{\bm\beta}$ of the regression parameters, the null hypothesis is 
$$H_0: \ep[Y - \hat{\bm{\beta}}^\top X\,|\, X] = 0 \quad \text{a.s.}$$
where $X \in \rr^d$ and $Y$ is a univariate random variable, \ie, $Z=(Y,X)$.
The data are generated from the data generating process (DGP):
\begin{equation*}
    Y = \bm{\beta}_0^\top X + e .
\end{equation*}
%where $g(\cdot):=\phi(\cdot)/\Phi(\cdot)$ is the Inverse Mill's ratio. In this case, $\phi(\cdot)$ and $\Phi(\cdot)$ are respectively the standard normal probability and cumulative density functions.
We set $\bm{\beta}_0 = \mathbf{1}$, and $X\sim\mathcal{N}(0,I_d)$.
For the error term $e$, we consider two scenarios:
\begin{enumerate*}[label=(\roman*)]
    \item \emph{Homoskedastic} (\texttt{HOM}): $e = \epsilon, \epsilon \sim \mathcal{N}(0,1)$ and 
    \item \emph{Heteroskedastic} (\texttt{HET}): $e = \epsilon\sqrt{0.1+0.1\|X\|_2^2}$.
\end{enumerate*}
In each trial, we obtain an estimate of $\bm{\beta}_0$ by $\hat{\bm{\beta}} = \bm{\beta}_0 + \gamma$ where $\gamma \sim \mathcal{N}(\mathbf{0},\delta^2 I_{d})$.
In this experiment, we set $d=5$.
When $\delta=0$, the CMR are fulfilled, whereas they are violated, \ie, $H_0$ is false, if $\delta \neq 0$.
Different values of $\delta$ correspond to different degrees of deviation from the null.

%%%
\vspace{-5pt}
\paragraph{Testing the simultaneous equation model (\texttt{SIMEQ}).}

Following \citet{Newey90:Instrumental} and \citet{Delgado06:Consistent}, we consider the equilibrium model
\begin{align*}
    Q &= \alpha_d P + \beta_d R + U, \quad \alpha_d < 0, && \text{(Demand)} \\
    Q &= \alpha_s P + \beta_s W + V, \quad \alpha_s > 0, && \text{(Supply)}
\end{align*}
where $Q$ and $P$ denote quantity and price, respectively, $R$ and $W$ are exogeneous variables, and $U$ and $V$ are the error terms. 
In this setting, $Z = (Q,P,R,W)$ and $X = (R,W)$.
The null hypothesis can be expressed as
\begin{equation*}
    H_0: \ep\left[\begin{array}{c|c}
         Q - \alpha_d P - \beta_d R & \multirow{2}{*}{X}  \\
         Q - \alpha_s P - \beta_s W &   
    \end{array}\right] = 
    \left[\begin{array}{c}
         0  \\
         0 
    \end{array}\right]
\end{equation*}
a.s.~for some $\theta_0 = (\alpha_d,\beta_d,\alpha_s,\beta_s)$.
We generate data according to $Q = \lambda_{11} R + \lambda_{12}W + V_1$ and $P = \lambda_{21} R + \lambda_{22}W + V_{2}$ where $R$ and $W$ are independent standard Gaussian random variables while $V_1$ and $V_2$ are correlated standard Gaussian random variables with $10^{-3}$ variance and $10^{-3}/\sqrt{2}$ covariance, and independent of $(R,W)$.
We set $(\lambda_{11},\lambda_{12},\lambda_{21},\lambda_{22}) = (1,-1,1,1)$ and provide the details on how to find the true parameters $\theta_0$ in Appendix \ref{sec:sem}.
The estimate $\hat{\theta}$ is obtained as in the previous experiment.
The null hypothesis corresponds to $\delta=0$ and different values of $\delta$ corresponds to alternative hypotheses. 
Rejecting $H_0$ means that the functional form of the supply and demand curves are misspecified.

%%%
Figure \ref{fig:results} depicts the empirical results for $n\in \{1,2,4,6,8,10\}\times 10^2$ and $\delta\in\{10^{-4},2\times 10^{-3},4\times 10^{-3},6\times 10^{-3},8\times 10^{-3},10^{-2}\}$.
First, it can be observed that KCM, ICM, and smooth tests are all capable of detecting the misspecification as the sample size and $\delta$ are sufficiently large.
Second, the KCM test tends to outperform both ICM and smooth tests in terms of the test power, especially in a low sample regime (see Figure \ref{fig:reg-hom-n}--\ref{fig:simeq-n}) and a small deviation regime (see Figure \ref{fig:reg-hom-d}--\ref{fig:simeq-d}).
In addition, the smooth test and the ICM test are competitive: there is no substantial evidence to conclude that one is always better than the other. 
Lastly, Figure \ref{fig:results-typei} in Appendix \ref{sec:type1-error} depicts that the Type-I errors of all tests are correctly controlled at $\alpha = 0.05$.

Lastly, we point out that this work does not elaborate on the effect of parameter estimation.
In practice, the candidate parameter $\hat{\theta}$ has to be estimated from the observed data, which changes the asymptotic distribution of the test statistic.
We envision the interplay between parameter estimation and hypothesis testing as an important arena for future work.

%%%%%
\section{CONCLUSION} 
\label{sec:conclusion}

To conclude, we propose a new conditional moment test called the KCM test whose statistic is based on a novel representation of the conditional moment restrictions in a reproducing kernel Hilbert space. 
This representation captures all necessary information about the original conditional moment restrictions.
Hence, the resulting test is consistent against all fixed alternatives, is easy to use in practice, and also has connections to existing tests in the literature.
It also has an encouraging finite-sample performance compared to those tests.
While the conditional moment restrictions have a long history in econometrics and so does the concept of reproducing kernel Hilbert spaces in machine learning, the intersection of these concepts remains unexplored.
We believe that this work gives rise to a new and promising framework for conditional moment restrictions which constitute numerous applications in econometrics, causal inference, and machine learning.

%Despite our promising results, there remain several open questions still to be answered and limitations to be overcome.
%First, it is crucial to consider the parameter estimation based on the CMR and understand how it affects the performance of the subsequent CM test.
%Second, it is natural to extend our framework via a general vector-valued RKHS which will allow for more flexibility in modelling the CMR.
%Third, an extension of our framework to semi-parametric and nonparametric settings will also make it more applicable to real-world econometric problems.
%Last but not least, we also plan to evaluate our framework on other realistic scenarios.

%Until now, our work provides evidences suggesting that reproducing kernels might be the suitable and flexible tools for analyzing central problems in economics, causal inference, and finance, and hence warrants further investigations.

%Although \eqref{eq:cond-moments} can be extended to a semi-parametric setting: $\ep[\bpsi(Z;\theta_0,h_0)|X] = \mathbf{0}$ where $h_0$ is an unknown nuisance function, we focus on the parametric setting throughout and refer the readers to \citet{AI03:CMR,Chen07:semi-non,Chen16:CMR} for semi-parametric and nonparametric settings.

% In the unusual situation where you want a paper to appear in the
% references without citing it in the main text, use \nocite

\subsubsection*{Acknowledgements}

We thank the anonymous reviewers for the helpful comments on our initial draft. 
KM is indebted to Motonobu Kanagawa and Junhyung Park for fruitful discussion.

{\small
\bibliography{refs}
\bibliographystyle{abbrvnat}
}

\newpage

\appendix
\onecolumn
\doparttoc \faketableofcontents
\addcontentsline{toc}{section}{Appendix} % Add the appendix text to the document TOC
\part{Appendix} % Start the appendix part
\parttoc % Insert the appendix TOC

\section{Conditional Moment Discrepancy (CMMD)}
\label{sec:cmd}

The maximum moment restriction (MMR) also allows us to compare two different models based on the conditional moment restriction (CMR).
Let $\mathcal{M}_{\theta_1}$ and $\mathcal{M}_{\theta_2}$ be two models parameterized by $\theta_1,\theta_2\in\Theta$, respectively.
Then, we can define a CMR-based discrepancy measure between these two models as follows.

%%%
\begin{definition}
    For $\theta_1,\theta_2\in\Theta$, a conditional moment discrepancy (CMMD) is defined as 
    $\Delta(\theta_1,\theta_2) := \| \mut_{\theta_1} - \mut_{\theta_2}\|_{\hbspf^p}.$
\end{definition}

By Theorem \ref{thm:identifiability}, $\Delta(\theta_1,\theta_2) \geq 0$ and $\Delta(\theta_1,\theta_2)=0$ if and only if the two models $\mathcal{M}_{\theta_1}$ and $\mathcal{M}_{\theta_2}$ are indistinguishable in terms of the CMR alone.  
Moreover, if the global identifiability \ref{asmp:a3} holds, $\Delta(\theta_0,\theta) = \MM(\theta)$ for all $\theta\in\Theta$.
Since 
\begin{equation*}
   \Delta(\theta_1,\theta_2) = \left\|\ep[\bm{\xi}_{\theta_1}(X,Z) - \bm{\xi}_{\theta_2}(X,Z)] \right\|_{\hbspf^q} = \left\|\ep[\bar{\bm{\xi}}(X,Z)] \right\|_{\hbspf^q}
\end{equation*}
where $\bar{\bm{\xi}}(x,z) := \bm{\xi}_{\theta_1}(x,z)-\bm{\xi}_{\theta_2}(x,z) = (\bpsi(z;\theta_1) - \bpsi(z;\theta_2))k(x,\cdot)$, the CMMD can be viewed as the MMR defined on a \emph{differential residual function} $\bpsi(z;\theta_1)-\bpsi(z;\theta_2)$.
As a result, $\Delta(\theta_1,\theta_2)$ also has a closed-form expression similar to that in Theorem \ref{thm:close-form}.

\begin{corollary}
    For $\theta_1,\theta_2\in\Theta$, let $$h((x,z),(x',z')) := (\bpsi(z;\theta_1)-\bpsi(z;\theta_2))^\top(\bpsi(z';\theta_1) - \bpsi(z';\theta_2))k(x,x')$$ and assume that
    $\ep[h((X,Z),(X,Z))] < \infty$. Then, we have 
    $\Delta^2(\theta_1,\theta_2) = \ep[h((X,Z),(X',Z'))]$
    where $(X',Z')$ is independent copy of $(X,Z)$ with an identical distribution.
\end{corollary}

\begin{proof}
    The result follows by applying the proof of Theorem \ref{thm:close-form} to the feature map $\bar{\bm{\xi}}(x,z) := \bm{\xi}_{\theta_1}(x,z)-\bm{\xi}_{\theta_2}(x,z) = (\bpsi(z;\theta_1) - \bpsi(z;\theta_2))k(x,\cdot)$.
\end{proof}

\noindent Furthermore, we can express the empirical CMMD as 
\begin{equation*}
\Delta^2_n(\theta_1,\theta_2) := \frac{1}{n(n-1)}\sum_{1\leq i\neq j\leq n} h((x_i,z_i),(x_j,z_j)) 
\end{equation*}
where $h((x_i,z_i),(x_j,z_j)) := (\bpsi(z_i;\theta_1)-\bpsi(z_i;\theta_2))^\top(\bpsi(z_j;\theta_1) - \bpsi(z_j;\theta_2))k(x_i,x_j)$.

As we can see, the RKHS norm, inner product, and function evaluation computed with respect to $\mut_\theta$ all have meaningful economic interpretations. 
Table \ref{tab:interpretation} summarizes these interpretations.

%%%% summary of operations
\begin{table}[h!]
    \centering
     \caption{Interpretations of different operations on $\mut_\theta$ in $\hbspf^q$.}
    \label{tab:interpretation}
    \begin{tabular}{cl}
         \toprule
         \textbf{Operation} & \textbf{Interpretation} \\
         \midrule
         $\|\mut_\theta\|_{\hbspf^q}$ & conditional moment violation  \\
         $\langle f, \mut_\theta\rangle_{\hbspf^q}$ & violation w.r.t. the instrument $f$ \\
         $\mut_\theta(x,z)$ & structural instability at $(x,z)$ \\
         $\|\mut_{\theta_1} - \mut_{\theta_2}\|_{\hbspf^q}$ & discrepancy between $\mathcal{M}_{\theta_1}$ and $\mathcal{M}_{\theta_2}$\\
         \bottomrule
    \end{tabular}
\end{table}

%%%%
\section{Parameter Estimation}
\label{sec:estimation}

Besides hypothesis testing, another important application of the CMR is parameter estimation. 
That is, given the CMR as in \eqref{eq:cond-moments}, we aim to find an estimate of $\theta_0$ that satisfies \eqref{eq:cond-moments} from the observed data $(x_i,z_i)_{i=1}^n$.
Based on the MMR, we define the estimator of $\theta_0$ as the parameter that minimizes \eqref{eq:estimator}:
\begin{equation}\label{eq:mmmr}
    \hat{\theta}_n := \arg\min_{\theta\in\Theta}\,\MH^2_n(\theta) = \arg\min_{\theta\in\Theta}\, \frac{1}{n(n-1)}\sum_{1\leq i\neq j \leq n} h_\theta((x_i,z_i),(x_j,z_j)).
\end{equation}
We call $\hat{\theta}_n$ a \emph{minimum maximum moment restriction} (MMMR) estimate of $\theta_0$. 
Note that it is also possible to adopt $V$-statistic in \eqref{eq:mmmr} instead of the $U$-statistic. 
Previously, \citet{Lewis18:AGMM} and \citet{Bennett19:DeepGMM} proposed to estimate $\theta_0$ based on \eqref{eq:sup-moments} and $\mathscr{F}$ that is parameterized by deep neural networks. However, their algorithms require solving a minimax game, whereas our approach for estimation is merely a minimization problem.

The following theorem shows that $\hat{\theta}_n$ is a consistent estimate of $\theta_0$. The proof can be found in Appendix \ref{sec:proof-consistency}.

\begin{theorem}[Consistency of $\hat{\theta}_n$]\label{thm:consistency}
    Assume that the parameter space $\Theta$ is compact. Then, we have $\hat{\theta}_n\xrightarrow{p}\theta_0$.
\end{theorem}

%\begin{theorem}[Asymptotic normality]\label{thm:normality}
%    If Assumptions [\textcolor{red}{TBA}] holds, then $\sqrt{n}(\thh_n - \theta_0)\xrightarrow{d}\mathcal{N}(\mathbf{0},\Sigma)$ where $\Sigma = \text{\textcolor{red}{[TBA]}}$.
%\end{theorem}

Despite the consistency, we suspect that $\hat{\theta}_n$ may not be asymptotically efficient and there exist better estimators.  
Theorem \ref{thm:expansion} shows that $\MM(\theta)$ depends on a continuum of moment conditions reweighted by the non-uniform eigenvalues $(\lambda_j)_j$, which suggests that a \emph{reweighting matrix} must also be incorporated in order to achieve the optimality \citep{Hall05:GMM}.
Constructing an optimal choice of reweighting matrix in an infinite dimensional RKHS is an interesting topic \citep{Carrasco00:Continuum}, and we leave it to future work.

%%%
\subsection{Maximum Moment Restriction for Instrumental Variable Regression}
\label{sec:ivr}

To illustrate one of the advantages of the MMR for parameter estimation, let us consider the nonparametric instrumental variable regression problem \citep{Angrist08:Harmless,Lewis18:AGMM,Singh19:KIV,Bennett19:DeepGMM,Muandet19:DualIV}.
Let $X$ be a treatment (endogeneous) variable taking values in $\inx\subseteq\rr^d$ and $Y$ a real-valued outcome variable.
Our goal is to estimate a function $g:\inx\to\rr$ from a structural equation model (SEM) of the form
\begin{equation}\label{eq:sem-iv}
    Y = g(X) + \varepsilon, \quad 
    X = h(Z) + f(\varepsilon) + \nu,
\end{equation}
\noindent where we assume that $\ep[\varepsilon] = 0$ and $\ep[\nu] = 0$.
Unfortunately, as we can see from \eqref{eq:sem-iv}, $\varepsilon$ is correlated with the treatment $X$, i.e., $\ep[\varepsilon|X] \neq 0$,  and hence standard regression methods cannot be used to estimate $g$.
This setting often arises when there exist unobserved confounders between the treatment $X$ and outcome $Y$.

In instrumental variable regression, we assume access to an \emph{instrumental} variable $Z$ which is associated with the treatments $X$, but not with the outcome variable $Y$, other than through its effect on the treatments.
Moreover, the instrument $Z$ is assumed to be uncorrelated with $\varepsilon$.
This implies the conditional moment restriction $\ep[\varepsilon\,|\,Z] = \ep[Y - g(X)\,|\,Z] = 0$ for $P_Z$-almost all $z$ \citep{Newey93:CMR,Lewis18:AGMM,Bennett19:DeepGMM}. 
Given an i.i.d. sample $(x_i,y_i,z_i)_{i=1}^n$ from $P(X,Y,Z)$, the MMR allows us to reduce the problem of estimating $g$ to a regularized empirical risk minimization (ERM) problem
\begin{eqnarray}\label{eq:kreg-obj}
\widehat{g}_{\lambda} &:=& \arg\min_{g\in\mathcal{G}_l}\;\MH^2_n(g)+\lambda \|g\|^2_{\mathcal{G}_l} \nonumber \\ 
    &=& \arg\min_{g\in\mathcal{G}_l}\;\frac{1}{n^2}\sum_{i=1}^n\sum_{j=1}^n (y_i - g(x_i))(y_j - g(x_j))k(z_i,z_j)+\lambda \|g\|^2_{\mathcal{G}_l},
\end{eqnarray}
where $\lambda$ is a positive regularization parameter and $\mathcal{G}_l$ is a reproducing kernel Hilbert space (RKHS) of real-valued functions on $\inx$ with the reproducing kernel $l:\inx\times\inx\to\rr$. 
Note that we adopt the $V$-statistic instead of the $U$-statistic in \eqref{eq:kreg-obj}.
By the representer theorem, the optimal solution to \eqref{eq:kreg-obj} can be expressed as a linear combination
\begin{equation}\label{eq:representer}
    \widehat{g}_{\lambda}(x) = \sum_{i=1}^{n}\alpha_i l(x,x_i)
\end{equation}
for some $(\alpha_1,\ldots,\alpha_n)\in\rr^n$.
Let $K =[k(z_i,z_j)]_{i,j}$ and $L =[l(x_i,x_j)]_{i,j}$ be the kernel matrices in $\rr^{n\times n}$ of $\bm{z} = [z_1,\ldots,z_n]^\top$ and $\bm{x} = [x_1,\ldots,x_n]^\top$, respectively, and $\bm{y} := [y_1,\ldots,y_n]^\top$.  
Substituting \eqref{eq:representer} back into \eqref{eq:kreg-obj} yields a \emph{generalized ridge regression} (GRR) problem
\begin{eqnarray}\label{eq:obj_reg}
    \bm{\alpha}_{\lambda} &:=& \arg\min_{\bm{\alpha}\in\rr^n}\;  \frac{1}{n^2}(\bm y-L\bm{\alpha})^{\top} K (\bm y-L\bm{\alpha})+\lambda \bm{\alpha}^{\top}L\bm{\alpha}.
\end{eqnarray}
That is, the optimal coefficients $\bm{\alpha}_{\lambda}$ can be obtained by solving the first-order stationary condition $(LKL + n^2\lambda L)\bm{\alpha} = LK\bm{y}$ and if $L$ is positive definite, the solution has a \emph{closed-form} expression, \ie,
\begin{equation}\label{eq:mmriv}
    \widehat{g}_{\lambda}(x) = \sum_{i=1}^{n}\alpha_{\lambda,i}l(x,x_i), \quad\quad \bm{\alpha}_{\lambda} = (LKL+ n^2\lambda L)^{-1}LK\bm{y}.
\end{equation}
Similar technique has been considered in \citet{Singh19:KIV} and \citet{Muandet19:DualIV}.
In \citet{Singh19:KIV}, the authors extended the two-stage least square (2SLS) by modeling the first-stage regression with the conditional mean embedding of $P(X|Z)$ \citep{Muandet17:KME} which is then used in the second-stage kernel ridge regression.
In \citet{Muandet19:DualIV}, the authors showed that the two-stage procedure can be reformulated as a convex-concave saddle-point problem. 
When the solutions lie in the RKHS, the closed-form solution similar to \eqref{eq:mmriv} and the one in \citet{Singh19:KIV} can be obtained.
By contrast, the MMR-based approach allows us to reformulate the problem directly as a generalized ridge regression (GRR) in which the values of hyperparameters, \eg, the regularization parameter $\lambda$, can be chosen via the popular cross-validation procedures.

%\subsection{Stochastic Optimization}

%Similar techniques for optimization have been used in %\citep{Barp19:MSDE}.

%\begin{proposition}\label{prop:info-matrix}
%    Information matrix field on $\Theta$.
%\end{proposition}

%\begin{equation}
%    \hat{\theta}_n^{t+1} = \hat{\theta}_n^t - \gamma_t XXXX
%\end{equation}

%\krik{Look at stochastic optimization on Riemmannian manifolds.}

%%%
\section{Experiments}

In this section, we provide further description of our experiments as well as additional experimental results.

%%%
\subsection{Simultaneous Equation Models}
\label{sec:sem}

A simultaneous equation model (SEM) is a fundamental concept in economics.
In one of our experiments, we consider the following SEM:
\begin{equation}\label{eq:sem}
    \begin{aligned}
    Q &= \alpha_d P + \beta_d R + U, \quad \alpha_d < 0, && \quad \text{(Demand)} \\
    Q &= \alpha_s P + \beta_s W + V, \quad \alpha_s > 0, && \quad \text{(Supply)}
\end{aligned}
\end{equation}
where $Q$ and $P$ denote quantity and price, respectively, $R$ and $W$ are exogeneous variables, and $U$ and $V$ are the error terms. 
To obtain \emph{reduced-form equations} of \eqref{eq:sem}, we must solve for the endogeneous variables $P$ and $Q$.
First, we solve for $P$ by equating the two equations in \eqref{eq:sem}:
\begin{equation}\label{eq:sem-p}
    P = \left[\frac{\beta_s}{\alpha_d - \alpha_s}\right]W - \left[\frac{\beta_d}{\alpha_d - \alpha_s}\right]R + \frac{V- U}{\alpha_d - \alpha_s}.
\end{equation}
Then, we can solve for $Q$ by plugging in $P$ to the supply equation in \eqref{eq:sem}:
\begin{equation}\label{eq:sem-q}
    Q = \left[\frac{\alpha_s\beta_s}{\alpha_d - \alpha_s} + \beta_s\right]W - \left[\frac{\alpha_s\beta_d}{\alpha_d-\alpha_s}\right]R + \frac{\alpha_s}{\alpha_d-\alpha_s}(V-U) + V .
\end{equation}
By comparing \eqref{eq:sem-p} and \eqref{eq:sem-q} to the data generating process in our experiment, we obtain the following system of equations:
\begin{equation}\label{eq:system}
\begin{aligned}
    \lambda_{11} &= -\frac{\alpha_s\beta_d}{\alpha_d - \alpha_s}, \\
    \lambda_{12} &= \frac{\alpha_s\beta_s}{\alpha_d-\alpha_s} + \beta_s,
\end{aligned}
\qquad \qquad 
\begin{aligned}
    \lambda_{21} &= -\frac{\beta_d}{\alpha_d - \alpha_s} \\
    \lambda_{22} &= \frac{\beta_s}{\alpha_d - \alpha_s} .
\end{aligned}
\end{equation}
Finally, setting $(\lambda_{11},\lambda_{12},\lambda_{21},\lambda_{22}) = (1,-1,1,1)$ and then solving \eqref{eq:system} result in a non-trivial solution $(\alpha_d,\beta_d,\alpha_s,\beta_s) = (-1,2,1,-2)$.
This solution coincides with the one obtained from the two-stage least square (2SLS) procedure \citep[Ch. 4]{Angrist08:Harmless}.

%%%%
\subsection{Type-I Errors}
\label{sec:type1-error}

The KCM test with bootstrapping is based on the asymptotic distribution of the test statistic under $H_0$ (cf.~Theorem \ref{thm:asymptotic}). Hence, the test reliably controls the Type-I error when the sample size is sufficiently large, \ie, we are in the asymptotic regime. For the considered examples, this is the case already for moderate sample sizes. 
We report the Type-I error at a significance level $\alpha=0.05$ for $n \in \{100,200,400,600,800,1000\}$ in Figure \ref{fig:results-typei}.

%%%%
\vspace{1.5em}
\begin{figure}[h!]
    \centering
    \begin{subfigure}[t]{0.33\textwidth}
        \centering
        \includegraphics[width=\textwidth]{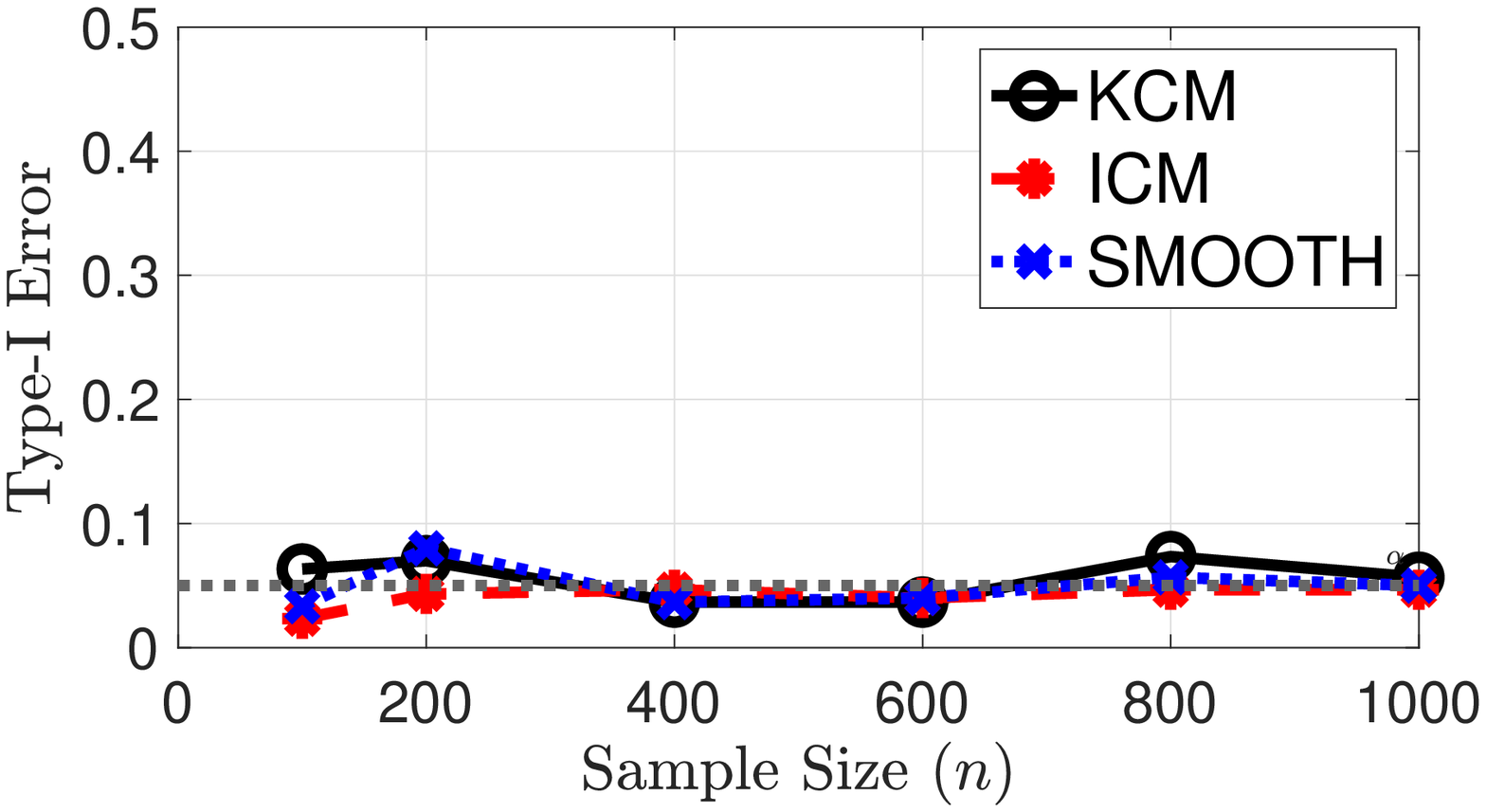}
        \caption{\texttt{REG-HOM}}
    \end{subfigure}%
    \hfill
    \begin{subfigure}[t]{0.33\textwidth}
        \centering
        \includegraphics[width=\textwidth]{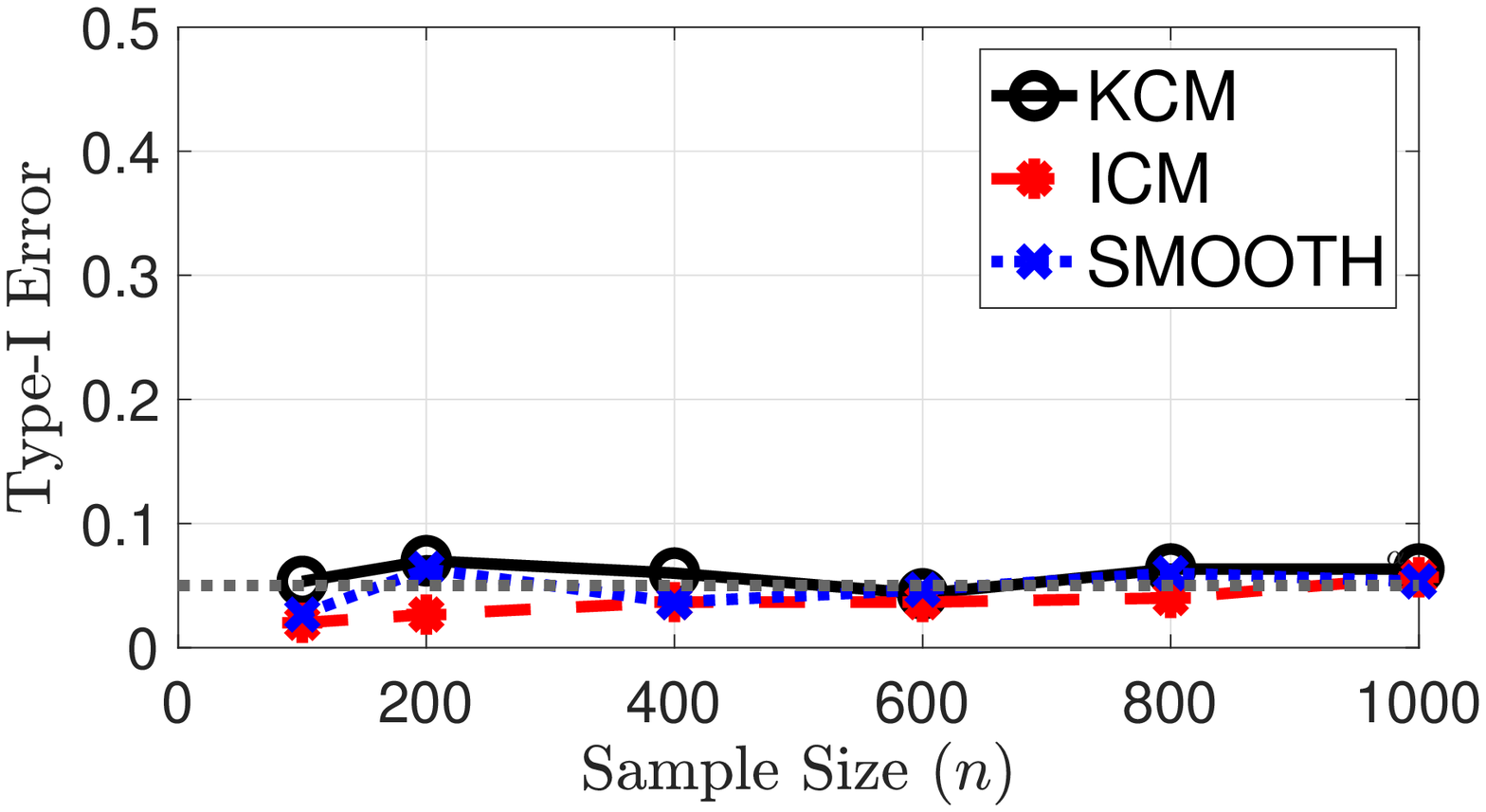}
        \caption{\texttt{REG-HET}}
    \end{subfigure}%
    \hfill
    \begin{subfigure}[t]{0.33\textwidth}
        \centering
        \includegraphics[width=\textwidth]{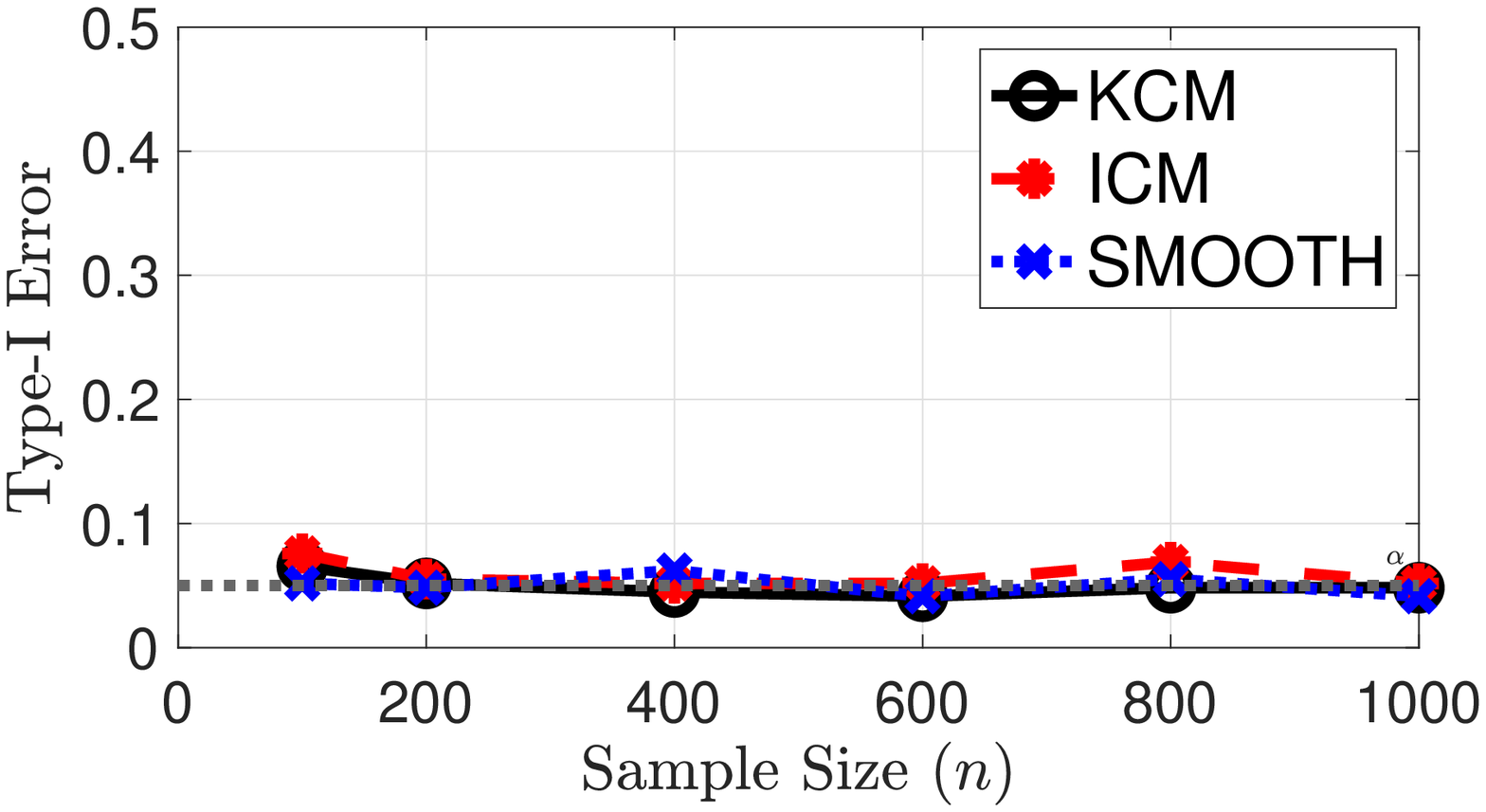}
        \caption{\texttt{SIMEQ}}
    \end{subfigure}%%%
    
    \caption{The Type-I errors averaged over 300 trials of KCM, ICM, and smooth tests under the null hypothesis ($\delta=0$) as we vary the sample size $n$.}
    \label{fig:results-typei}
\end{figure}

\section{Proofs}
\label{sec:proofs}

This section collects all the proofs of the results presented in the main paper.

\subsection{Proof of Lemma \ref{lem:optimal}}

\begin{proof}
     We have $M_{\theta_0} f = \sum_{i=1}^q \ep[\psi_i(Z;\theta_0)f_i(X)]$ and, for all $i=1,\ldots,q$, 
     \begin{equation*}
     \ep_{\mathit{XZ}}[\psi_i(Z;\theta_0)f_i(X)] = \ep_X[\ep_{Z}[\psi_i(Z;\theta_0)f_i(X)|X]] 
    = \ep_X[\ep_Z[\psi_i(Z;\theta_0)|X]f_i(X)] = 0 
    \end{equation*}
    by the law of iterated expectation. The last equality follows from the definition of $\theta_0$ and the continuity of $f_i$, \ie, by Assumption \ref{asmp:a4}.
\end{proof}

\subsection{Proof of Theorem \ref{thm:cmme-consistency}}

Our result follows directly from \citet[Lemma 2]{Smale07:Theory} and \citet[Theorem 3.4]{Pinelis94:Banach} which rely on the Bennett inequality for vector-valued random variables.
We reproduce the proof here for completeness.
\begin{proof}
    First, recall that $\mut_{\theta} = \ep[\bxi_{\theta}(X,Z)]$ and $\muth_\theta = \frac{1}{n}\sum_{i=1}^n\bxi_\theta (x_i,z_i)$ for the independent random variables $\{\bxi_\theta(x_i,z_i)\}_{i=1}^n$. 
    Then, for any $\varepsilon > 0$, it follows from \citet[Lemma 1]{Smale07:Theory} that
    \begin{equation*}
        P\left\{\left\|\frac{1}{n}\sum_{i=1}^n\bxi_\theta(x_i,z_i) - \mut_\theta\right\|_{\hbspf^q} \geq \varepsilon \right\} \leq 2\exp\left\{ -\frac{n\varepsilon}{2C_\theta}\log\left(1 + \frac{C_\theta \varepsilon}{\sigma_\theta^2}\right)\right\} .
    \end{equation*}
    Taking $t := C_\theta \varepsilon/\sigma_\theta^2$ and applying the inequality $\log(1+t) \geq t/(1+t)$ for all $t>0$ yield 
    \begin{eqnarray*}
        P\left\{ \left\| \frac{1}{n}\sum_{i=1}^n\bxi_\theta (x_i,z_i) - \mut_\theta \right\|_{\hbspf^q} \geq \varepsilon \right\} &\leq& 2\exp\left\{ -\frac{n\varepsilon}{2C_\theta}\left(\frac{C_\theta \varepsilon}{C_\theta\varepsilon + \sigma_\theta^2}\right)\right\} \\
        &=& 2\exp\left\{ -\frac{n\varepsilon^2}{2C_\theta\varepsilon + 2\sigma_\theta^2}\right\}.
    \end{eqnarray*}
    The value of $\varepsilon > 0$ for which this probability equal to $\delta$ can be obtained by solving the quadratic equation $n\varepsilon^2 = \log(2/\delta)(2C_\theta\varepsilon + 2\sigma_\theta^2)$.
    As a result, we have with confidence $1-\delta$ that
    \begin{equation}
        \left\| \frac{1}{n}\sum_{i=1}^n\bxi_\theta(x_i,z_i) - \mut_\theta\right\|_{\hbspf^q} \leq \frac{2C_\theta\log\frac{2}{\delta}}{n} + \sqrt{\frac{2\sigma_\theta^2\log\frac{2}{\delta}}{n}},
    \end{equation}
    as required.
\end{proof}

It remains to show that, for each $\theta\in\Theta$, there exists a constant $C_\theta < \infty$ such that $\|\bxi_\theta(X,Z)\|_{\hbspf^q} < C_\theta$ almost surely. Note that for any $(x,z)\in\inx\times\inz$ for which $P_{\mathit{XZ}}(x,z) > 0$,
\begin{eqnarray*}
    \|\bm{\xi}_\theta(x,z)\|_{\hbspf^p} &=& \sqrt{\|\bm{\xi}_\theta(x,z)\|_{\hbspf^p}^2} \\ 
    &=& \sqrt{\bpsi(z;\theta)^\top\bpsi(z;\theta)k(x,x)} \\ 
    &\leq& \sup_{x,z}\sqrt{\bpsi(z;\theta)^\top\bpsi(z;\theta)k(x,x)} < \infty,
\end{eqnarray*}
where the last inequality follows from Assumptions \ref{asmp:a2} and \ref{asmp:a4}.

\subsection{Proof of Theorem \ref{thm:identifiability}}
\label{sec:proof-identify}

\begin{proof}
    If $\CMR(x;\theta_1) = \CMR(x;\theta_2)$ for $P_X$-almost all $x$, then the equality $\mut_{\theta_1} = \mut_{\theta_2}$ follows straightforwardly.
    Suppose that $\mut_{\theta_1} = \mut_{\theta_2}$ and let $\bm{\delta}(x) := \CMR(x;\theta_1) - \CMR(x;\theta_2)$.
    Then, we have 
    \begin{eqnarray}
    \|\mut_{\theta_1} - \mut_{\theta_2}\|^2_{\hbspf^q} &=& \left\| \int \bxi_{\theta_1}(x,z)\dd P_{XZ}(x,z) - \int \bxi_{\theta_2}(x,z)\dd P_{XZ}(x,z) \right\|_{\hbspf^q}^2 \nonumber \\ 
    &=& \left\| \int \CMR(x;\theta_1)k(x,\cdot) \dd P_X(x) - \int \CMR(x;\theta_2)k(x,\cdot)\dd P_X(x) \right\|_{\hbspf^q}^2 \nonumber \\
    &=& \left\| \int (\CMR(x;\theta_1) - \CMR(x;\theta_2) )k(x,\cdot)\dd P_X(x) \right\|_{\hbspf^q}^2 \nonumber \\
    &=& \iint \bm{\delta}(x)^\top k(x,x')\bm{\delta}(x')\dd P_X(x)\dd P_{X'}(x') = 0, \label{eq:integral}
    \end{eqnarray}
    where $X'$ is an independent copy of $X$.
    It follows from \eqref{eq:integral} and Assumption \ref{asmp:a2} that the function $g(x) := \bm{\delta}(x)p_X(x)$ has zero L2-norm, \ie, $\|g\|_2^2 = 0$ where $p_X$ denotes the density of $P_X$.
    As a result, $\bm{\delta}(x) = \mathbf{0}$ a.e. $P_X$ implying that $P_X(B_0)=1$ where $B_0 := \left\{ x\in\inx \, :\, \CMR(x;\theta_1) - \CMR(x;\theta_2) = \mathbf{0}\right\}$. Therefore, $\CMR(x;\theta_1) = \CMR(x;\theta_2)$ for $P_X$-almost all $x$. This completes the proof.
\end{proof}

\subsection{Proof of Theorem \ref{thm:close-form}}

\begin{proof}
    By the definition of $\mathbb{M}(\theta)$ and the Bochner integrability of $\bxi_\theta$,
    \begin{eqnarray*}
    \mathbb{M}^2(\theta) &=& 
    \|\mut_\theta\|_{\hbspf^q}^2 \\
    &=& \langle \mut_\theta,\mut_\theta\rangle_{\hbspf^q} \\
    &=& \langle \ep[\bm{\xi}_\theta(X,Z)], \ep[\bm{\xi}_\theta(X,Z)] \rangle_{\hbspf^q} \\
    &=& \ep[\langle \bm{\xi}_\theta(X,Z), \ep[\bm{\xi}_\theta(X,Z)] \rangle_{\hbspf^q}] \\
    &=& \ep[\langle \bm{\xi}_\theta(X,Z), \bm{\xi}_\theta(X',Z') \rangle_{\hbspf^q}] \\
    &=& \ep[h_\theta((X,Z),(X',Z'))],
    \end{eqnarray*}
    where $(X',Z')$ is an independent copy of $(X,Z)$ with an identical distribution.
\end{proof}

\subsection{Proof of Theorem \ref{thm:expansion}}

\begin{proof} 
    By Mercer's theorem \citep[Theorem 4.49]{SteChr2008}, we have $k(x,x') = \sum_j\lambda_je_j(x)e_j(x')$ where the convergence is absolute and uniform.
    Recall that $\bm{\zeta}_\theta^j(x,z) := \left(\psi_1(z;\theta)e_j(x), \ldots,\psi_q(z;\theta)e_j(x) \right)$.
    Hence, we can express the kernel $h_\theta$ as
    \begin{align*}
        h_\theta((x,z),(x',z')) &= \bpsi(z;\theta)^\top\bpsi(z';\theta)k(x,x') \\
            &= \bpsi(z;\theta)^\top\bpsi(z';\theta)\left(\sum_j\lambda_je_j(x)e_j(x')\right) \\
            &= \sum_j\lambda_j\bpsi(z;\theta)^\top\bpsi(z';\theta)e_j(x)e_j(x') \\
            &= \sum_j\lambda_j\left[\bpsi(z;\theta)e_j(x) \right]^\top \left[\bpsi(z';\theta)e_j(x')\right] \\
            &= \sum_j\lambda_j\bm{\zeta}_\theta^j(x,z)^\top \bm{\zeta}_\theta^j(x',z').
            \intertext{Since $\lambda_j > 0$, the function $h_\theta$ is positive definite. Then, we can express $\MM^2(\theta)$ as follows:}
            \MM^2(\theta) &= \ep[h_\theta((X,Z),(X',Z'))] \\
            &= \ep\left[\sum_j\lambda_j\bm{\zeta}_\theta^j(X,Z)^\top \bm{\zeta}_\theta^j(X',Z')\right] \\
            &= \sum_j\lambda_j\ep_{XZ}\left[\bm{\zeta}_\theta^j(X,Z)\right]^\top \ep_{X'Z'}\left[\bm{\zeta}_\theta^j(X',Z')\right] \\
            &= \sum_j\lambda_j \left\|\ep_{XZ}\left[\bm{\zeta}_\theta^j(X,Z)\right]\right\|_2^2.
    \end{align*}
    This completes the proof.
\end{proof}

\subsection{Proof of Theorem \ref{thm:consistency}}
\label{sec:proof-consistency}

In order to show the consistency of $\hat{\theta}_n := \arg\min_{\theta\in\Theta}\,\MH^2_n(\theta)$, we need the uniform consistency of $\MH^2_n(\theta)$ and the continuity of $\theta\mapsto \MM^2(\theta)$.
The following lemma gives these two results.

%%%
\begin{lemma}\label{lem:mmr-consistency}
    Assume that there exists an integrable and symmetric function $F_{\bpsi}$ such that $\|\bpsi(z,\theta)\|_2 \leq F_{\bpsi}(z)$ for any $\theta\in\Theta$ and $z\in\inz$.
    If Assumption \ref{asmp:a4} holds, $\sup_{\theta\in\Theta}|\mathbb{M}_n^2(\theta)-\mathbb{M}^2(\theta)| \xrightarrow{p} 0 $ and $\theta \mapsto \MM^2(\theta)$ are continuous.
\end{lemma}
\begin{proof}
    Recall that 
    \begin{eqnarray*}
        \MM^2(\theta) &=& \ep[h_\theta((X,Z),(X',Z'))] \\ 
        \MH_n^2(\theta) &=& \frac{1}{n(n-1)}\sum_{1\leq i\neq j \leq n} h_\theta((x_i,z_i),(x_j,z_j)),
    \end{eqnarray*}
    where $h_{\theta}((x,z),(x',z')) = \langle \bxi_{\theta}(x,z),\bxi_{\theta}(x',z') \rangle_{\hbspf^q} = \bpsi(z;\theta)^\top\bpsi(z';\theta)k(x,x')$.
    Then, it follows that
    \begin{align*}
        \left|h_\theta((x,z),(x',z'))\right| &= \left| \langle \bxi_{\theta}(x,z),\bxi_{\theta}(x'z') \rangle_{\hbspf^q} \right| \\
        &\leq \|\bxi_{\theta}(x,z)\|_{\hbspf^q}\cdot \|\bxi_{\theta}(x',z')\|_{\hbspf^q} && \\
        &= \sqrt{\bpsi(z;\theta)^\top\bpsi(z;\theta)k(x,x)}\sqrt{\bpsi(z';\theta)^\top\bpsi(z';\theta)k(x',x')} \\
        &= \|\bpsi(z;\theta)\|_2\|\bpsi(z';\theta)\|_2\sqrt{k(x,x)k(x',x')} \\
        &\leq F_{\bpsi}(z)F_{\bpsi}(z')\sqrt{k(x,x)k(x',x')}, &&
    \end{align*}
    where $F_{\bpsi}$ is an integrable and symmetric function.
    By Assumption \ref{asmp:a4}, $(x,x')\mapsto \sqrt{k(x,x)k(x',x')}$ is also an integrable function.
    Hence, $h_\theta$ is integrable.
    Since $\Theta$ is compact, it then follows from \citet[Lemma 2.4]{Newey94:LS} that $\sup_{\theta\in\Theta}|\MH_n^2(\theta)-\mathbb{M}^2(\theta)| \xrightarrow{\text{p}} 0 $ and $\theta \mapsto \MM^2(\theta)$ is continuous.
\end{proof}

Now, we are in the position to present the proof of Theorem \ref{thm:consistency}.

\begin{proof}[Proof of Theorem \ref{thm:consistency}]
    By Assumption \ref{asmp:a3} and Theorem \ref{thm:identifiability}, $\MM^2(\theta) = 0$ if and only if $\theta=\theta_0$.
    Thus $\MM^2(\theta)$ is uniquely minimized at $\theta_0$. 
    Since $\Theta$ is compact, $\MM^2(\theta)$ is continuous and $\MH_n^2(\theta)$ converges uniformly in probability to $\MM^2(\theta)$ by Lemma \ref{lem:mmr-consistency}. 
    Then, $\hat{\theta}_n \xrightarrow{\text{p}} \theta_0$ by \citet[Theorem 2.1]{Newey94:LS}.
\end{proof}

\subsection{Proof of Theorem \ref{thm:asymptotic}}

\begin{proof}
First, we need to check that $\sigma_h^2 \neq 0$ when $\theta \neq \theta_0$ and $\sigma_h^2 = 0$ when $\theta=\theta_0$. 
Then, the results follow directly from \citet[Sec. 5.5.1 and Sec. 5.5.2]{serfling1980approximation}.

Note that $\ep_{u'}[h_\theta(u,u')] = \ep_{u'}[\langle \bxi_\theta(u),\bxi_\theta(u')\rangle_{\hbspf^q}] = \langle \bxi_\theta(u),\ep_{u'}[\bxi_\theta(u')]\rangle_{\hbspf^q}=\langle \bxi_\theta(u) ,\mut_\theta\rangle_{\hbspf^q} = M_\theta \bxi_\theta(u)$.
When $\theta=\theta_0$, it follows that $\ep_{u'}[h_{\theta_0}(u,u')]=0$ by Lemma \ref{lem:optimal}, and hence $\sigma_h^2 = 0$.

Next, suppose that $\theta\neq \theta_0$.
Then, $\ep_{u'}[h_\theta(u,u')] = M_{\theta}\bxi_\theta(u) =: c(u)$. 
Since $\sigma_h^2 = \text{Var}_u[c(u)] = \ep_u[ (c(u) - \ep_{u'}[c(u')] )^2 ]$, $\sigma_h^2=0$ if and only if $c(u)$ is a constant function.
Note that we can write $c(u) = c(x,z) = \ep_{X'Z'}[\bpsi(Z';\theta)^\top\bpsi(z;\theta)k(x,X')]$. 
Therefore, by Assumptions \ref{asmp:a3} and \ref{asmp:a4}, $c(u)$ cannot be a constant function, implying that $\sigma_h^2 > 0$.
\end{proof}

\subsection{Proof of Theorem \ref{thm:kcm-icm}}

\begin{proof}
    Since the kernel $k(x,x')=\varphi(x-x')$ is a shift-invariant kernel on $\rr^d$, it follows from Theorem \ref{thm:bochner} that
    $$\varphi(x-x') = (2\pi)^{-d/2}\int_{\rr^d} e^{-i(x-x')^\top\omega} \dd\Lambda(\omega).$$
    Therefore, we can express $\MM^2(\theta)$ as
    \begin{align*}
        \MM^2(\theta) &= \ep[\bpsi(Z;\theta)^\top\bpsi(Z';\theta)k(X,X')] \\
            &= \ep[\bpsi(Z;\theta)^\top\bpsi(Z';\theta)\varphi(X-X')] \\
            &= (2\pi)^{-d/2}\ep\left[\bpsi(Z;\theta)^\top\bpsi(Z';\theta)\left(\int_{\rr^d} e^{-i(X-X')^\top\omega} \dd\Lambda(\omega)\right) \right] \\
            &= (2\pi)^{-d/2}\ep\left[\bpsi(Z;\theta)^\top\bpsi(Z';\theta)\left(\int_{\rr^d} e^{-i\omega^\top X} \cdot e^{i\omega^\top X'} \dd\Lambda(\omega)\right) \right] \\
            &= (2\pi)^{-d/2}\ep\left[\int_{\rr^d}\bpsi(Z;\theta)^\top\bpsi(Z';\theta) e^{-i\omega^\top X} e^{i\omega^\top X'} \dd\Lambda(\omega) \right] \\
            &= (2\pi)^{-d/2}\ep\left[\int_{\rr^d}\left[\bpsi(Z;\theta)e^{-i\omega^\top X} \right]^\top\left[\bpsi(Z';\theta)e^{i\omega^\top X'}\right]  \dd\Lambda(\omega) \right] \\
            &= (2\pi)^{-d/2}\int_{\rr^d} \ep\left[\bpsi(Z;\theta)e^{-i\omega^\top X} \right]^\top \ep\left[\bpsi(Z';\theta)e^{i\omega^\top X'}\right]  \dd\Lambda(\omega) \\
            &= (2\pi)^{-d/2}\int_{\rr^d} \left\|\ep\left[\bpsi(Z;\theta)\exp(i\omega^\top X) \right]\right\|_2^2 \dd\Lambda(\omega).
    \end{align*}
    This completes the proof.
\end{proof}

%%% 
%\section{Supplementary Results}

%Additional experimental results

%\subsection{KCM Test Statistic with Radial Kernels}
%\label{sec:kcm-radial}

%\begin{theorem}[Schoenberg’s theorem]
%    A continuous function $\varphi: [0,\infty)\to\rr$ is positive definite and radial on $\rr^d$ for all $d$ if and only if it is of the form \begin{equation*}
%        \varphi(r) = \int_0^\infty e^{-r^2\omega^2} \dd\Lambda(\omega)
%    \end{equation*}
%    where $\Lambda$ is a finite non-negative Borel measure on $[0,\infty)$.
%\end{theorem}

%\begin{corollary}
%Let $k(x,x')$ be a radial kernel, \ie, $k(x,x') = \varphi(\|x-x'\|)$ for some positive definite function $\varphi:[0,\infty)\to\rr$.
%Then we have
%$$\MM^2(\theta) = \int_0^{\infty} (\ep[\bpsi(Z;\theta)\exp(-\omega^2r^2)])^2 \dd\Lambda(\omega)$$
%where $\Lambda$ is a finite non-negative Borel measure on $[0,\infty)$.
%\end{corollary}

%\begin{proof}
%    XXX
%\end{proof}

%Similarly, with radial kernels, we obtain the ICM test proposed in \citet{Bierens90:FuncForm}.

%%%%%%%%%%%%%%%%%%%%%%%%%%%%%%%%%%%%%%%%%%%%%%%%%%%%%%%%%%%%%%%%%%%%%%%%%%%%%%%
%%%%%%%%%%%%%%%%%%%%%%%%%%%%%%%%%%%%%%%%%%%%%%%%%%%%%%%%%%%%%%%%%%%%%%%%%%%%%%%

%%%

\end{document}